\documentclass[11pt]{amsart}
\usepackage{latexsym}
\usepackage{amssymb,amsmath,mathabx, amscd}
\usepackage[pdftex]{graphicx}
\usepackage{enumerate}
\usepackage{tikz}
\usepackage[margin=1in]{geometry}
\usepackage{listings}
\usepackage{courier}
\usepackage{tikz-cd}
\usepackage{color}
\usepackage{MnSymbol}
\usepackage{hyperref}
\usepackage{mathrsfs, colonequals}

\newcommand{\cc}{\mathbb{C}}

\newcommand{\pp}{\mathbb{P}}

\newcommand{\I}{\mathcal{I}}

\renewcommand{\O}{\mathcal{O}}

\newcommand{\F}{\mathcal{F}}

\newcommand{\res}{\operatorname{res}}

\renewcommand{\Im}{\operatorname{Im}}

\newcommand{\rk}{\operatorname{rk}}
\newcommand{\ev}{\operatorname{ev}}
\newcommand{\genus}{\operatorname{genus}}

\renewcommand{\bar}{\overline}

\newtheorem{thm}{Theorem}[section]
\newtheorem{ithm}{Theorem}
\newtheorem{icor}[ithm]{Corollary}
\newtheorem{lem}[thm]{Lemma}

\newtheorem{prop}[thm]{Proposition}
\newtheorem{cor}[thm]{Corollary}

\newcommand{\defi}[1]{\textsf{#1}} 

\theoremstyle{definition}
\newtheorem{defin}[thm]{Definition}

\theoremstyle{remark}
\newtheorem{rem}{Remark}

\begin{document}

\title[Interpolation for curves in $\pp^4$]{Interpolation for Brill-Noether curves in $\pp^4$}
\author{Eric Larson}
\address{Department of Mathematics, Stanford University, Stanford, CA 94305}
\email{elarson3@gmail.com}
\author{Isabel Vogt}
\address{Department of Mathematics, Massachusetts Institute of Technology, Cambridge, MA 02139}
\email{ivogt@mit.edu}
\date{\today}

\maketitle

\begin{abstract}
In this paper, we compute the number of general points through which a general Brill--Noether curve in $\pp^4$ passes.
We also prove an analogous theorem when some points
are constrained to lie in a transverse hyperplane.  As explained in \cite{over}, these results play an essential role in the first author's proof of the Maximal Rank Conjecture \cite{mrc}.
\end{abstract}

\section{Introduction}

In this paper we deal with the fundamental incidence question: through how many general points can one pass a curve of degree $d$ and genus $g$ in $\pp^r$?  The natural conjecture would be that each point imposes independent conditions, and therefore the maximal number of general points should be the dimension of the scheme parameterizing such curves divided by $r-1$.
We focus here on curves of general moduli, where $\rho(d,g,r) = (r+1)d - rg - r(r+1) \geq 0$ and there is a unique component $\bar{M}_g(\pp^r, d)^\circ$ of the Kontsevich space dominating $\bar{M}_g$,
which is of dimension
\[(r + 1) d - (r - 3)(g - 1).\]
Curves of this component are therefore expected to pass through
\[f(d, g, r) := \left\lfloor \frac{(r + 1)d - (r - 3)(g - 1)}{r - 1}\right\rfloor\]
general points.
We call curves parameterized by this component \defi{Brill--Noether curves} (BN-curves).  Previous work has settled the cases where $d \geq g + r$ \cite{joint}, and where $r=3$ \cite{vogt}.

One can also consider the intersection of a general Brill--Noether curve
of degree $d$ and genus $g$ with a transverse hypersurface $S$ of degree $n$ in $\pp^r$.
For exactly five values of $(r, n)$ --- namely, $(r, n) \in \{(2,1),(2,2),(3, 1), (3, 2), (4, 1)\}$ --- this
intersection is a general set of $dn$ points on $S$
for all but finitely many $(d, g)$.
In these three cases, it is therefore natural to ask a stronger question: can
one pass a general Brill--Noether curve of degree $d$ and genus $g$
through $f(d, g, r)$ points which are general subject to the constraint that
$dn$ of them lie on a transverse hypersurface of degree $n$?
This question is trivial for $(r,n) \in \{(2,1), (2,2)\}$, and work of \cite{vogt} and \cite{quadrics} answer this stronger question
for $(r, n) = (3, 1)$ and $(r, n) = (3, 2)$,
so the only remaining case is $(r, n) = (4, 1)$.

We answer these questions completely in the case $r=4$
(respectively $(r, n) = (4, 1)$).  Along with results of \cite{vogt, quadrics} discussed above, these form base cases for inductive arguments used to prove the Maximal Rank Conjecture, as outlined in \cite{over}.

As in previous work, we deduce these answers by studying the normal bundle $N_C$ (respectively the twisted normal bundle $N_C(-1)$) for $C$ a general Brill--Noether curve.  We say that a vector bundle $E$ on a curve $C$ \defi{satisfies interpolation} if it is nonspecial ($h^1(E) =0$) and for a general effective divisor $D$ of any degree, either 
\[ h^0(E(-D)) = 0\text{, or } h^1(E(-D))=0. \]
The main result is then:

\begin{ithm}\label{main}
Let $C$ be a general Brill--Noether curve of degree $d$ and genus $g$ in $\pp^4$.  Then $N_C(-1)$ satisfies interpolation if and only if
\[(d,g) \notin \{(6,2), (8,5), (9,6), (10,7)\}.\]

Furthermore, if $D$ is an effective divisor of degree $d - 1$
supported in a general hyperplane section, then $N_C(-D)$ satisfies interpolation
if and only if
\[(d, g) \neq (6, 2).\]
\end{ithm}

As a corollary, we deduce the answer to the second
stronger question posed above when $(r, n) = (4, 1)$:

\begin{icor}\label{cor_twist}
A general Brill--Noether curve of degree $d$ and genus $g$ in $\pp^4$
passes through a set of $f(d, g, 4)$ points which are general subject to
the constraint that $d$ of them lie on a transverse hyperplane,
if and only if
\[(d,g) \notin \{(8,5), (9,6), (10,7)\}.\]
The geometry of these three cases is as follows:
\begin{itemize}
\item 
In these three cases, the statement fails because
the section by a transverse hyperplane
of the general
curve is not a general set of $d$ points in the hyperplane;
instead it is a set of $d$ points which are general subject
to the condition that they are distinct points
which lie on a complete intersection of $11 - d$ quadrics.

\item 
Nevertheless, for $(d, g) \in \{(8, 5), (10, 7)\}$ --- but not for $(d, g) = (9, 6)$ --- a general
Brill--Noether curve of this degree and genus
still passes through $f(d, g, 4) + 1$ points
which are general subject to the constraint that $d$ of them lie
on a transverse hyperplane and are distinct points which lie
on a complete intersection of $11 - d$ quadrics.

\item
This fails for $(d, g) = (9, 6)$ because a general Brill--Noether curve of this degree and genus
lies on a quartic del Pezzo surface in $\pp^4$. However, such 
a curve
still passes through $f(d, g, 4)$ points
which are general subject to either the constraint that $d$ of them lie
on a transverse hyperplane and are distinct points
which lie on a complete intersection of $11 - d$ quadrics,
or the constraint that $d - 1$ of them lie on a transverse hyperplane.
\end{itemize}
\end{icor}

Using Theorem~\ref{main}, some additional work in the cases
$(d, g) \in \{(8, 5), (9, 6), (10, 7)\}$
yields:

\begin{ithm} \label{notwist}
Let $C$ be a general Brill--Noether curve of degree $d$ and genus $g$ in $\pp^4$.
Then $N_C$ satisfies interpolation if and only if
\[(d,g) \neq (6,2).\]
\end{ithm}

In turn, we deduce an answer to the first question posed earlier
for $r = 4$:

\begin{icor}\label{cor_notwist}
A general Brill--Noether curve of degree $d$ and genus $g$ in $\pp^4$
passes through a set of $f(d, g, 4)$ general points, with no exceptions.
\end{icor}


The main techniques in this paper resemble those in \cite{joint} and \cite{quadrics}, namely we degenerate to reducible curves.  As interpolation is an open condition, it suffices to find one BN-curve of each relevant degree and genus whose (twisted) normal bundle satisfies interpolation.  We achieve this by specializing to nodal curves in the boundary of $\bar{M}_g(\pp^4, d)^\circ$, where we can relate interpolation for the normal bundle of the union to interpolation for some modifications of the normal bundles of the components.  
The degenerations here, though, are necessarily more complicated than those in \cite{joint}, as degenerating to the union of a curve and a $1$- or $2$-secant line alone cannot hope to apply to special curves where $d < g + r$.
We employ three primary classes of degenerations.  

First, as developed in \cite{quadrics}, we degenerate to a BN-curve which is the union of a curve $X \subset H \simeq \pp^3$ in a hyperplane and a nonspecial curve $Y$, transverse to $H$, meeting $X$ transversely at a finite set of specified points.  This allows us to leverage the known interpolation results in the nonspecial range \cite{joint} and in $\pp^3$ \cite{vogt} to inductively prove interpolation for all BN-curves from a finite set of base cases.  Resolving these base cases consumes the remainder of the paper.

The second type of degeneration is to a connected nodal curve containing a canonical curve $D \subset \pp^4$ as well as three $2$-secant lines to $D$.  Using the geometry of the canonical curve, in particular the fact that $N_D(-1) \simeq K_D^{\oplus 3}$, we prove a delicate degeneration lemma to reduce to the problem of interpolation for the normal bundle of the union to interpolation for a simple modification of the normal bundle of the remainder.

The final specialization is to the union of a degenerate rational normal quartic curve meeting the other component in $6$ points.  The rational normal quartic curve is the union of three $2$-secant lines, and the unique fourth line meeting all three.  
We hope that this technique is robust, and that some variant may be useful for related problems.

The paper is organized as follows.  In Section~\ref{prelim} we briefly recall basic facts about modifications of vector bundles on curves and the technique of degenerating to reducible curves to prove interpolation.  In Section~\ref{finitelist} we employ the first degeneration to the union of a curve in a hyperplane and a curve transverse to the hyperplane to reduce to checking a finite set of cases.  The canonical curve degeneration lemma is in Section~\ref{canonical} and the degeneration with the rational quartic curve is in Section~\ref{quartic}.  At this point all of the difficult work has been done, and we collect loose ends in Section~\ref{proofs} to prove the main theorems.

\subsection*{Acknowledgments}
We would like to thank Joe Harris for his guidance, as well as members of the MIT and Harvard math departments for helpful conversations.  The first author acknowledges the generous support both of the Fannie and John Hertz Foundation, and
of the Department of Defense (NDSEG fellowship). The second author acknowledges the generous support of the National Science Foundation Graduate Research Fellowship Program under Grant No. 1122374.

\section{Preliminaries}\label{prelim}

\subsection{Notation}

In this paper, we deal with the Kontsevich space compactification of Brill--Noether curves of degree $d$ and genus $g$ in $\pp^r$.  In particular, for the remainder of the paper, a curve $C \to \pp^r$ is shorthand for a stable map $f \colon C \to \pp^r$.  In all of our computations, we will assume that the map $f$ is unramified, so that the normal sheaf of the map, which we will denote by $N_C$ when the map $f$ is implicit, or $N_f$ to draw attention to the map, 
is a vector bundle and its cohomology controls the deformation theory of $f$.  
A general such map is an embedding of a smooth curve, and so $N_f$ is isomorphic to the normal bundle of the embedded smooth curve.  The only subtlety arises when degenerating such curves.

Suppose that $f \colon C \to \pp^r$ and $g \colon D \to \pp^r$ are curves of degree and genus $(d,g)$ and $(d', g')$ respectively, and $\{p_i\}_{i \in I} \in C$ and $\{q_i\}_{i \in I} \in D$ are points such that 
\[f(p_i) = g(q_i),  \text{ for all $i \in I$}.\]
Write $\Gamma = f(\{p_i\}) = g(\{q_i\})$.  Let $C \cup_\Gamma D$ denote the nodal curve obtained by gluing $p_i \in C$ to $q_i \in D$, mapped to $\pp^r$ via the stable map
\[f\cup g \colon C \cup_\Gamma D \to \pp^r.\]
If $C \to \pp^r$ and $D \to \pp^r$ are curves passing through a finite set of points $\Gamma \subset \pp^r$, then we use $C \cup_\Gamma D$ to denote the same construction for some choice of $\{p_i\}$ and $\{q_j\}$ mapping isomorphically onto $\Gamma$ under $f$ and $g$ respectively.

Note that the curve $C \cup_\Gamma D$ is of degree $d + d'$ and genus $g + g' + \#I -1$ even if $f(C)$ and $g(D)$  meet away from $\Gamma$.  This is our primary interest in using the Kontsevich space compactification.

\subsection{Modifications of Vector Bundles}\label{modifications}

In this section we recall some basic facts about modifications of vector bundles on curves to streamline the coming computations.  For a more detailed exposition, see \cite[\textsection 2-3]{joint}.

\begin{defin}
Let $E$ be a vector bundle on a variety $X$ and let $D \subset X$ be an effective Cartier divisor.  Suppose that $U \subset X$ is an open subset containing $D$ and that $F \subseteq E|_U$ is a subbundle defined on $U$.  Then the \defi{modification of $E$ at $D$ towards $F$}, denoted $E[D \to F]$, is defined by the exact sequence
\[0 \to E[D \to F] \to E \to (E/F)|_D \to 0. \]
\end{defin}
The reason for the terminology is that sections of the modification $E[D \to F]$  ``point towards $F$" when restricted to the divisor $D$:
\[H^0(E[D \to F]) = \{ \sigma \in H^0(E) : \sigma|_D \in F|_D \}. \]
The modified bundle $E[D \to F]$ is isomorphic to $E$ when restricted to $X \smallsetminus D$.  Therefore given another divisor $D' \subset X \smallsetminus D$ and another subbundle $F' \subseteq E$, we may form the multiple modification
\[E[D \to F][D' \to F'] \]
taking our open set $U = X \smallsetminus D$.  In this setting $E[D \to F][D' \to F'] \simeq E[D' \to F'][D \to F]$;
however we will need multiple modifications in a stronger setting in which $D$ and $D'$ are allowed to meet. For that reason, we have the following definition:

\begin{defin}
Let $\{F_i \subseteq E|_{U_i}\}_{i \in I}$ be a finite collection of subbundles of $E$ defined on open neighborhoods $U_i$ of a point $x \in X$.
We say that this collection is \defi{tree-like at $x \in X$} if for all index subsets $I' \subseteq I$, either
\begin{itemize}
\item The fibers $\{F_i|_x \}_{i \in I'} \subset E|_x$ are linearly independent, or
\item There are distinct indices $i, j \in I'$, and an open $U \subset U_i \cap U_j$ containing $x$ such that $F_i|_U \subset F_j|_U$.
\end{itemize}
\end{defin}

\begin{defin}
A \defi{modification datum} $M$ for a vector bundle $E$ is a finite collection of tuples $\{(D_i, U_i, F_i)\}_{i \in I}$, where $D_i$ is an effective Cartier divisor on $X$, and $U_i \subset X$ is an open subset containing $D_i$, and $F_i \subseteq E|_{U_i}$ is a subbundle.  We say that a modification datum is \defi{tree-like} if for all $x \in X$, the collection of subbundles $\{F_i \subseteq E : x \in D_i\}$ is tree-like at $x$.
\end{defin}

The key result is then the following

\begin{prop}[{\cite[Proposition 2.17]{joint}}]\label{transfer}
There is a bijection $\varphi$ between modification data $M$ for $E$ such that $(D, U, F) \cup M$ is tree-like and modification data $M'$ for $E[D \to F]$ such that $(D, U, F) \cup M'$ is tree-like.  This bijection is compatible with pullbacks and restricts to the identity when $D = \emptyset$.  
\end{prop}
We may therefore recursively define $E[M]$, for $M = (D, U, F) \cup N$ by
\[E[M] \colonequals E[D \to F][\varphi(N)]. \]
For simplicity, when $M = \{(D_i, U_i, F_i)\}_{1 \leq i \leq n} $, we will write 
\[E[M] = E[D_1\to F_1] \cdots [D_n \to F_n]. \]
We have the following nice properties of tree-like modifications:

\begin{prop}[{\cite[Propositions 2.20, 2.21, 2.23]{joint}}] \label{mod-prop}
Let $E$ be a vector bundle on $X$ and let $M$ be a tree-like modification datum.  Then we have
\begin{enumerate}
\item (Commuting modifications) If $M'$ is any reordering of $M$, we have $E[M] \simeq E[M']$.
\item (Commuting with twists) If $D$ is any effective Cartier divisor, then $E[M](D) \simeq E(D)[M]$.
\item (Combining modifications to the same bundle) If $D$ and $D'$ are effective
Cartier divisors, we have
\[E[D \to F][D' \to F] \simeq E[D + D' \to F]. \]
\item (Combining modifications at the same divisor) If $F_1$ and $F_2$ are independent, then
\[E[D \to F_1][D \to F_2] \simeq E[D \to F_1 \oplus F_2](-D). \]
\end{enumerate}
\end{prop}

\begin{rem}
Even though we focus in this paper on modifications of vector bundles on curves,
for which the curve-to-projective extension theorem easily
gives the analog of Proposition \ref{transfer} for arbitrary collections of modifications (i.e.\ if we drop the tree-like assumption on both sides),
we still need the language of tree-like modifications for two reasons.

First, the properties of Proposition~\ref{mod-prop} will not, in general, hold
when the modifications are not tree-like.

Second, 
we need the more general formalism of modifications on
arbitrary varieties in order to ``limit" various modifications together by considering vector bundles on the total space of a family of curves, see Proposition~\ref{limit} below for an illustrative example.  Compatibility with base-change guarantees that the results agrees with the expected result on each curve. 
\end{rem}

We will deal primarily with modifications of the normal bundles of curves $C \to \pp^r$ towards pointing bundles $N_{C \to \Lambda}$, for $\Lambda \subset \pp^r$ a linear space of dimension $\gamma$.  To recall, suppose that the locus 
\[U_{C, \Lambda} \colonequals \{p \in C : T_pC \cap \Lambda = \emptyset\} \]
is dense and contains the singular locus of $C$.  Then on $U_{C, \Lambda}$, we suggestively define $N_{C \to \Lambda}|_U$ to be the kernel of the map
\[N_C|_{U_{C, \Lambda}}\to N_{\pi_\Lambda|_{U_{C, \Lambda}}}, \]
where $N_{\pi_\Lambda|_{U_{C, \Lambda}}}$ is the normal sheaf of the projection map restricted to $U_{C, \Lambda}$.  By our assumptions on $U_{C, \Lambda}$, the curve-to-projective extension theorem implies that there is a unique vector subbundle $N_{C\to \Lambda} \subseteq N_C$ agreeing on $U_{C, \Lambda}$ with $N_{C \to \Lambda}|_{U_{C, \Lambda}}$.

When modifying towards a pointing bundle, we will write
\[N_C[D \to \Lambda] \colonequals N_C[D \to N_{C \to \Lambda}]. \]

These modifications naturally arise when computing the restrictions of normal bundles of nodal curves to the irreducible factors, as in the following fundamental result of Hartshorne and Hirschowitz.

\begin{lem}[{\cite{HH}}]\label{normal_bundle_nodal}
Let $C_1 \cup_\Gamma C_2$ be a nondegenerate nodal curve in $\pp^r$, with $\Gamma = \{p_1, \cdots p_n\}$.  For each $i$, let $p_i'$ be a choice of point on the tangent line $T_{p_i}C_1$ and let $p_i''$ be a choice of point on the tangent line $T_{p_i} C_2$.  Then we have
\begin{align*}
N_{C_1 \cup_\Gamma C_2}|_{C_1} &\simeq N_{C_1}(\Gamma)[p_1 \to p_1''] \cdots [p_n \to p_n''] \\
N_{C_1 \cup_\Gamma C_2}|_{C_2} &\simeq N_{C_2}(\Gamma)[p_1 \to p_1'] \cdots [p_n \to p_n']
\end{align*}
\end{lem}

As indicated above, we will make ample use of the fact that the semicontinuity theorem guarantees that interpolation for a vector bundle $E$ on $C$ follows from interpolation for some specialization of $E$.  We make this precise in the following example, but will omit these details in the future.

\begin{prop}\label{limit}
Let $C \to \pp^r$ be a curve and $p_1$ and $p_2$ distinct points on $C$.  Let $q_1$ and $q_2$ be distinct points in $\pp^r$.  Then 
\begin{enumerate}
\item if $p_2$ is general and $T_{p_1}C, q_1, q_2$ span a $\pp^3$, interpolation for 
\[N_C[p_1 \to q_1][p_1 \to q_2] \simeq N_C(-p_1)[p_1 \to q_1 + q_2]\]
implies interpolation for $N_C[p_1 \to q_1][p_2 \to q_2]$,
\item if $q_2$ is general and $T_{p_2}C, q_1$ span a $\pp^2$, interpolation for
\[N_C[p_1 \to q_1][p_2 \to q_1] \simeq N_C[p_1 + p_2 \to q_1]\]
implies interpolation for $N_C[p_1 \to q_1][p_2 \to q_2]$.
\end{enumerate}
\end{prop}
\begin{proof}
For part (1), let $B = C$ and let $\Delta \subset B \times C$ denote the diagonal.  Write $P_i$ for the divisor $B \times \{p_i\} \subset B \times C$.  Let $\pi_1 \colon B \times C \to B$ and $\pi_2 \colon B \times C \to C$ denote the two projections.  The we have inclusions of vector bundles on $B \times C$:
\[\pi_2^*N_{C \to q_1} \subseteq \pi_2^*N_{C}, \qquad \pi_2^*N_{C \to q_2} \subseteq \pi_2^*N_{C}.\]
By our assumption that $T_{p_1}C, q_1, q_2$ are in linear general position, the modification datum 
\[M = \{(P_1, B \times C, \pi_2^*N_{C \to q_1} ), (\Delta, B \times C, \pi_2^*N_{C \to q_2}) \} \]
is tree-like at $\Delta \cap P_1 = (p_1, p_1) \in B \times C$ and hence is tree-like.  Therefore the multiple modification
\[\pi_2^*N_C[M] = \pi_2^*N_C[P_1 \to q_1][\Delta \to q_2] \]
is a vector bundle on $B \times C$, which restricts to $N_C[p_1 \to q_1][p_2 \to q_2]$ on the fiber $\pi_1^{-1}(p_2)$ and $N_C[p_1 \to q_1][p_1 \to q_1]$ on the fiber $\pi_1^{-1}(p_1)$.   By the semicontinuity theorem, interpolation for the second implies interpolation for the first with $p_2$ general, proving part (1).

For part (2), let $B =\pp^r$.
It is shown in \cite[Section 5]{joint} that as the linear space $q_2$ varies in $B$, the open subsets $U_{C, q_2} \subset C$ fit together into an open set $U_{B \times C, q_2} \subset B \times C$.  Furthermore, over this open subset the pointing bundles $N_{C \to q_2}$ fit together into a vector bundle, which we will abusively also denote $N_{C \to q_2}$, on $U_{B \times C, q_2}$.  By assumption, $T_{p_2}C \cap q_1 = \emptyset$, so $(q_1, p_2) \in U_{B \times C, q_2}$, and therefore after shrinking $B$ we may assume that the divisor $P_2 \subset U_{B \times C, q_2}$.  Therefore
\[(P_2, U_{C \times B, q_2}, N_{C \to q_2}) \]
is a modification datum for $\pi_2^*N_C$, and the modification data
\[M = \{(P_1, B \times C, \pi_2^* N_{C \to q_1}), (P_2, U_{B \times C, q_2}, N_{C \to q_2})\} \]
is vacuously tree-like.  Therefore the multiple modification $\pi_2^*N_C[M]$ is a vector bundle on $U_{B \times C, q_2}$ which restricts to 
\[N_C[p_1 \to q_1][p_2 \to q_1] \]
over the special point $q_1 \in B$ and to $N_C[p_1 \to q_1][p_2\to q_2]$ for general $q_2$ in a neighborhood of $q_1$ in $B$.  Interpolation for the second therefore follows from interpolation for the first by the semicontinuity theorem.
\end{proof}

In the above proof, the generality assumptions on the points in $C$ and linear subspaces in $\pp^r$ guaranteed that the necessary modifications on the total space of the family of curves were well-defined and tree-like.  In the remainder of the paper, when we wish to specialize either points or linear spaces to prove interpolation, we will indicate what genericity assumptions we are using, but leave the reader to fill in the details of the argument with the above example as a guide.

\subsection{Interpolation for Vector Bundles}

In the course of the proof, we will make use of the following lemma, which allows us to prove interpolation by finding a single divisor such that the twist of our vector bundle by that divisor has a vanishing cohomology group.  As in \cite{vogt}, if the Euler characteristic $\chi(E)$ is divisible by the rank $r$ of $E$, it suffices to find a single effective divisor of degree $\chi(E)/r$ such that $h^0(E(-D)) = 0$, or equivalently $h^1(E(-D)) = 0$.  The following generalization covers all remaining possibilities for a vector bundle of rank $3$.

\begin{lem}\label{check_one}
Let $E$ be a nonspecial vector bundle of rank $r$ on a curve.  Assume that $\chi(E) \equiv 1 \pmod r$ (respectively $\chi(E) \equiv -1 \pmod r$). If there exists an effective divisor $D$ of degree $\left\lfloor \frac{\chi(E)}{r} \right\rfloor$ (resp.\ $\left\lceil \frac{\chi(E)}{r} \right\rceil$) such that $h^1(E(-D)) = 0$ (resp.\ $h^0(E(-D)) = 0$), then $E$ satisfies interpolation.
\end{lem}
\begin{proof}
By \cite[Proposition 4.6]{joint} it suffices to find one effective divisor $D_+$ of degree $\left\lceil \frac{\chi(E)}{r} \right\rceil$ such that $h^0(E(-D_+))=0$ and one effective divisor $D_-$ of degree $\left\lfloor \frac{\chi(E)}{r} \right\rfloor$ such that $h^1(E(D_-)) = 0$.  By assumption we have one of these; we will show that we can construct the other by adding or subtracting a point.

Let $D$ be the effective divisor on $C$ specified in the statement of the lemma.  As $E$ is assumed to be nonspecial, the following sequence
\[0 \to H^0(E(-D)) \to H^0(E) \to E|_D \to H^1(E(-D)) \to 0 \]
is exact. 

Suppose first that $\chi(E) \equiv 1 \pmod{r}$, so we can take $D_- = D$. We take $D_+ = D + p$
for $p \in C$ general.
As $h^1(E(-D)) = 0$, we must have that $h^0(E(-D)) = 1$, and so $E(-D)$ has a unique global section.  As such a section does not vanish at a general point $p \in C$, the twist $h^0(E(-D-p)) = 0$ as desired.

Suppose next that we are in the case $\chi(E) \equiv -1 \pmod r$, so we can take $D_+ = D$.  Then 
\[H^0(E) \hookrightarrow E|_D \simeq \bigoplus_{p \in D} E|_p\]
is the inclusion of a hyperplane in a vector space of dimension $\chi(E) +1$.  As the subspaces $E|_p$ for $p \in D$ generate $E|_D$, there is some $p \in D$ such that $E|_p \not\subset H^0(E)$.  Therefore the composition with the projection map
\[H^0(E) \to E|_D \to E|_{D \smallsetminus p} \]
is surjective, and hence $h^1(E(-D+p)) = 0$.
\end{proof}

In the course of the proof of Theorem~\ref{main}, we will need the following more general definition of interpolation for the space of sections of a vector bundle.

\begin{defin}
Let $E$ be a vector bundle of rank $r$ on a curve $C$ and let $V \subset H^0(C, E)$ be some subspace of its space of global sections.  For points $p_1, \ldots, p_n \in C$, we will abbreviate 
\[ V(-p_1-\cdots -p_n) \colonequals \{ \sigma \in V : \sigma|_{p_1} = \cdots =  \sigma|_{p_n} = 0 \}.\]
Then we say that $V$ satisfies interpolation if for general points $p_1, \ldots, p_n \in C$, 
\[\dim V(-p_1-\cdots-p_n) = \max(0, \dim V - rn). \]
\end{defin}
\begin{rem}
With this definition, a vector bundle $E$ satisfies interpolation if and only if $\chi(E) = h^0(E)$ and its space of global sections $H^0(E)$ satisfies interpolation.
\end{rem}

\begin{lem}\label{h0_gluing}
Let $C = X \cup_\Gamma Y$ be a reducible nodal curve and let $E$ be a vector bundle on $C$.  Suppose the restriction map
\[\res_{X,\Gamma} \colon H^0(E|_X) \to E|_\Gamma \]
is injective.  If the space of sections
\[V \colonequals \{\sigma \in H^0(E|_Y) : \sigma|_\Gamma \in \Im(\res_{X,\Gamma}) \},\]
has dimension $\chi(E)$ and satisfies interpolation, then $E$ satisfies interpolation as well.
\end{lem}
\begin{proof}
By assumption we have that restriction to $Y$ gives an isomorphism
\begin{equation}\label{isom} H^0(C, E) \xrightarrow{\simeq} V. \end{equation}
For any divisor $D \subset Y \smallsetminus \Gamma$, the isomorphism \eqref{isom} restricts to an isomorphism $H^0(E(-D)) \to V(-D)$; therefore interpolation for $E$ follows from interpolation for $V$ and the fact that $\chi(E) = \dim V$. \end{proof}
\begin{cor}\label{h0_gluing_vb}
With the setup of Lemma \ref{h0_gluing}, if there exists a vector bundle $F$ on $Y$ with
\[H^0(F) = V, \qquad \chi(F) = \chi(E), \qquad \rk(F) = \rk(E), \]
then if $F$ satisfies interpolation as a vector bundle on $Y$, we have that $E$ satisfies interpolation as a vector bundle on $C$.
\end{cor}

The main techniques in Sections \ref{canonical} -- \ref{quartic} will be degeneration to nodal curves containing lines which are 1-, 2-, or 3-secant to the remainder of the curve.  We give here some results on the bundle $F$ of Corollary \ref{h0_gluing_vb} which will be used there.

In order to deal with modified normal bundles we will set up the following notation.  Suppose that $X \to \pp^r$ is a nodal curve, and $L$ is a secant line to $X$ meeting $X$ in a set $S$.  Let $N_{X\cup_S L}'$ be a vector bundle on the union $X \cup_S L$ that is isomorphic to $N_{X \cup_S L}$ on some open subset $U \subset X \cup_S L$ containing all of $L$.
Then we will write $N_X'$ for the bundle on $X$ which is $N_{X \cup_S L}'|_{X\smallsetminus S}$ glued along $U$ to $N_{X \cup_S L}|_X$.

\begin{lem} \label{1-secant}
Let $X \to \pp^r$ be a curve and $p \in X$ a smooth point.  Let $L \subset \pp^r$ be a line through $p$ distinct from the tangent line to $X$ at $p$,
and denote by $C = X \cup_p L$ the nodal curve which is the union.  Let $x \in L$ be another point and suppose that $x' \in \pp^r$ is a point not contained
in the tangent plane to $C$ at $p$.  Then if the bundle 
\[N_X'(-1)(p)[p \to x][p \to x+x'] \]
satisfies interpolation (as a vector bundle on $X$), then $N_C'(-1)[x \to x']$ satisfies interpolation on $C$.
\end{lem}
\begin{proof}
This follows from Lemma~8.5 of \cite{joint}
taking $p_1 = x$ and $\Lambda_1 = x'$;
and taking $p_2 = \O_C(1)|_L$ and $\Lambda_2 = \emptyset$.
\end{proof}
%

We now consider the case $r = 4$,
and prove three different degeneration results for the normal bundle $N_C$ and the twisted normal bundle $N_C(-1)$.

\begin{lem}\label{2-secant-easy}
Let $X \to \pp^4$ be a curve, $p$ and $q$ be smooth points of $X$, and $x'$ be a point in $\pp^r$.
Suppose that the tangent lines to $X$ at $p$ and $q$, together with $x'$,
span $\pp^4$.
Let $L \subset \pp^r$ be a $2$-secant line through $p$ and $q$.  Let $x$ and $y$ be points on $L$ distinct from $p$ and $q$.  Denote by $C$ the union $X \cup_{p,q} L$.   If the bundle
\[N_X'(-1)[p \to q][q\to p] \]
satisfies interpolation, then the bundle $N_C'(-1)(-y)[x \to x']$, and hence $N_C'(-1)[x \to x']$, satisfies interpolation.
\end{lem}
\begin{proof}
Write $p'$ and $q'$ for choices of points on $T_p C$ and $T_q C$
distinct from $p$ and $q$ respectively.  By assumption $p'$, $q'$ and $x'$ span $\pp^4$, we may therefore choose $3$ independent hyperplanes $H_1$, $H_2$, and $H_3$ from the net defining the line $L$ such that $p' \in H_2 \cap H_3$, $q' \in H_1 \cap H_3$, and $x' \in H_1 \cap H_2$.  With the choice of the $H_i$, the normal bundle of $L$ is simply
\begin{equation}\label{ci_line} N_L \simeq N_{H_1}|_L \oplus N_{H_2}|_L \oplus N_{H_3}|_L. \end{equation}
Consider now the bundle
\[N_C'(-1)(-y)[x \to x']|_L
\simeq N_L(-1)(p+q-y)[p \to p'][q \to q'][x \to x']. \]
As $p' \in H_2 \cap H_3$, with respect to the decomposition in \eqref{ci_line}, the modification at $p$ towards $p'$ is simply a modification towards $N_{H_1}|_L$ at $p$.  Therefore this modification is the simple twist
\[N_L[p \to p'] \simeq N_{H_1}|_L \oplus N_{H_2}|_L(-p) \oplus N_{H_3}|_L(-p). \]
If we use the analogous expressions for the modifications at $q$ and $x$, we obtain the isomorphism
\begin{align*}
N_C&'(-1)(-y)[x \to x']|_L \\
&\simeq N_{H_1}|_L(-1)(p+q-y -q-x) \oplus N_{H_2}|_L(-1)(p+q-y-p-x) \oplus N_{H_3}|_L(-1)(p+q-y-p-q) \\
& \simeq \O_L(p-y -x) \oplus \O_L(q-y-x) \oplus \O_L(-y) \\
& \simeq \O_L(-1)^{\oplus 3}.
\end{align*}
As the bundle $\O(-1)$ on $\pp^1$ has no global sections, the restriction map to $p$ and $q$ is trivially injective and the desired bundle $F$ from Corollary \ref{h0_gluing_vb} is therefore
\[N_C'(-1)(-y)[x \to x']|_X (-p-q) \simeq N_X'(-1)[p \to q][q \to p], \]
which has the correct rank and Euler characteristic.
\end{proof}

\begin{lem}\label{3-secant}
Let $X$ be a curve in $\pp^4$ and let $L$ be a $3$-secant line meeting $X$ at smooth points $o, p, q$ such that the tangent lines $T_oX$, $T_pX$, and $T_qX$ span $\pp^4$.  Let $C \colonequals X \cup_{o +p+q} L$.  If $N_C'|_X$ satisfies interpolation, then $N_C'$ satisfies interpolation.
\end{lem}
\begin{proof}
Write $o'$, $p'$, and $q'$ for points on $T_oX$, $T_pX$, and $T_qX$
respectively, distinct from $o$, $p$, and $q$ respectively.
By our assumption on the tangent directions, we have  
\[N_C'|_L \simeq N_L(o+p+q)[o \to o'][p \to p'][q \to q'] \simeq \O_L(2)^{\oplus 3}.\]
Restriction to the three points $o+p+q$ yields an isomorphism on global sections
\[H^0(\O_L(2))^{\oplus 3} \to \O_L(2)^{\oplus 3}|_{o + p+ q}. \]
Therefore the bundle $F$ of Corollary \ref{h0_gluing_vb} is simply $N_C'|_X$, which is
of the correct rank and Euler characteristic.
\end{proof}

\begin{lem}\label{3-secant-twist}
Let $X$ be a curve in $\pp^4$ and let $L$ be a $3$-secant line meeting $X$
at smooth points $o, p, q$ such that the tangent lines $T_oX$, $T_pX$,
and $T_qX$ span $\pp^4$.  Let $C \colonequals X \cup_{o +p+q} L$.
Let $x$ and $y$ be points on $L$ disjoint from $o$, $p$, and $q$.
If $N_C'(-1)|_X(-o-p-q)$ satisfies interpolation, then $N_C'(-1)(-x-y)$,
and hence $N_C'(-1)$, satisfies interpolation.
\end{lem}
\begin{proof}
Analogously as in Lemma \ref{3-secant}, our assumptions give 
\[N_C'(-1)|_L \simeq \O_L(1)^{\oplus 3}, \]
and so $N_C'(-1)(-x-y)|_L \simeq \O_L(-1)^{\oplus 3}$, which has no global sections.  Therefore, the bundle $F$ of Corollary \ref{h0_gluing_vb} is simply
$N_C'(-1)(-x-y)|_X(-o-p-q) \simeq N_C'(-1)|_X(-o-p-q)$,
which has the correct rank and Euler characteristic.
\end{proof}

\section{Reduction to Finite List}\label{finitelist}

In this section we reduce the proof of Theorem~\ref{main}
to a finite set of cases.  To make the notation less burdensome, we make the following definition:


\begin{defin}
A pair of integers $(d,g)$ is called a \defi{Brill-Noether} pair (BN-pair) if $\rho(d,g,4) \geq 0$.  Furthermore, we say that a BN-pair $(d,g)$ is \defi{good} if for a general BN-curve $C$ of degree $d$ and genus $g$, the twisted normal bundle $N_C(-1)$ satisfies interpolation.
\end{defin}

The main result is the following lemma:

\begin{lem}\label{cut_to_finite}
Let $(d,g)$ be a BN-pair which is not in $\{(6,2), (8,5), (9,6), (10,7) \}$.  Then $(d,g)$ is good provided that all pairs $\{(9,5), (11,8), (12,10), (13,10), (13,11), (14,12) \} $ are good.
\end{lem}

In order to prove this result, we will leverage the interpolation result in $\pp^3$ \cite[Theorem 1.1]{vogt} and in the nonspecial range in $\pp^4$ \cite{joint}.  More specifically, we specialize to a reducible curve $X \cup_\Gamma Y$, where $X$ is contained in a hyperplane
$H \simeq \pp^3$, and $Y$ is transverse to the hyperplane and meets $X$ in a finite set of
points $\Gamma$ which are general in $H$.  The key result here is the following:

\begin{lem}\label{in_transverse}
Let $s$ be a positive integer, and
$X \to H$ be a BN-curve of degree $d'' \geq s + 1$
and genus $g'' := 2s + 1 - d''$
in a hyperplane $H \subset \pp^4$.  Let $\Gamma$ be a finite set of $s$ general points on $X$.  Let $Y$ be a BN-curve of degree and genus $(d', g')$ passing through $\Gamma$ with tangent direction transverse to the hyperplane $H$.  Let $C = X \cup _{\Gamma} Y$
(which by construction is of degree $d = d' + d''$ and genus $g = g' + 3s - d''$).
If both $N_{X/H}(-1)$ and $N_Y(-1)$ satisfy interpolation, then so does $N_{C}(-1)$.
\end{lem}
\begin{proof}
Note that
\[\frac{\chi(N_{C}(-1))}{\rk N_C(-1)} = \frac{2d + 1 - g}{3} = d'' - s + \frac{2d' + 1 - g'}{3} \geq d'' - s.\]
Let $E$ and $F$ be general effective divisors on $Y$ and $X$ respectively,
of degrees $e$ and $f = d'' - s$ respectively.
We will show that if $e \leq \frac{2d' + 1 - g'}{3}$, then $h^1(N_C(-1)(-E-F)) = 0$; and if $e \geq \frac{2d' + 1 - g'}{3}$, then $h^0(N_C(-1)(-E-F)) = 0$.

By Lemma 2.7 of \cite{quadrics}, to show the appropriate vanishing of $h^i$ it suffices to check the three vanishings
\begin{equation}\label{hi}
h^i(N_{X/H}(-1)(-\Gamma-F)) = 0, \qquad h^i(\O_X(\Gamma - F)) = 0, \qquad h^i(N_Y(-1)(-E)) =0.
\end{equation}
In the first case, $N_{X/H}(-1)$ satisfies interpolation by assumption, and so $N_{X/H}(-1)(-\Gamma - F)$ has no cohomology as $\Gamma + F$ is general of degree $d''$.  In the second case, $\O_X(\Gamma - F)$ is a general line bundle of degree $g'' - 1$ (since $d'' \geq s + 1$ implies $s + (d'' - s) \geq 2s + 1 - d'' = g''$), and so also has no cohomology by Riemann-Roch.  Finally, the last vanishing is exactly the statement that $N_Y(-1)$ satisfies interpolation.
\end{proof}
To find suitable $X$ and $Y$ as above, we use the following elementary lemma, which follows from the fact that if the degree of a line bundle $L$ is at least $d+g$, then $L(-1)$ is effective.
\begin{lem}
Let $E$ be a vector bundle of rank $k$ satisfying interpolation on a curve $C \in \pp^r$ of degree $d$ and genus $g$.  If 
\[\chi(E) \geq k(d + g),\]
then $E(-1)$ also satisfies interpolation.
\end{lem}
\begin{cor}\label{basecases}
If $C$ is a general BN-curve in $\pp^3$ (resp. $\pp^4$) of degree $d$ and genus $g$ with $d \geq g$ (resp. $d \geq 2g$), then 
\[N_C \text{ satisfies interpolation } \qquad \Leftrightarrow \qquad N_C(-1) \text{ satisfies interpolation.} \]
\end{cor}

We will now inductively prove Lemma \ref{cut_to_finite} by attaching curves $X \in \pp^3$ to curves $Y \in \pp^4$, for which $N_Y(-1)$ is already known to satisfy interpolation.  We will only use the following curves for $X$:
\begin{center}
\begin{tabular}{c | c | c | c}
$d''$ & $g''$ & $s$ & $\Delta(\rho)$ \\ \hline
9 & 6 & 7 & -3 \\
3 & 0 & 1 & 15\\
6 & 3 & 4 & 6\\
7 & 4 & 5 & 3\\
8 & 5 & 6 & 0
\end{tabular}
\end{center}
In order to conclude from Lemma \ref{in_transverse} that for a general $C \to \pp^4$ of degree $d' + d''$ and genus $g' + 3s - d''$, the bundle $N_C(-1)$ satisfies interpolation, we must show that $X \cup_{\Gamma} Y$ is a BN-curve.  In all of our cases this ends up being easily deduced from previous results.

First we use $X = (9,6)$ to reduce to checking small genus:

\begin{lem}\label{dplus9}
If a BN-pair $(d',g')$ is good and $d' \geq 7$, then the pair $(d' + 9, g' + 12)$ is good (when it is BN).
\end{lem}
\begin{proof}
Let $Y \to \pp^4$ be a general BN-curve of degree $d'$ and genus $g'$,
and $H$ be a general hyperplane.
Since $(d', g')$ is good by assumption,
the hyperplane section $Y \cap H$ is general;
in particular, since $d' \geq 7$,
it contains a set $\Gamma$ of $7$ general points
which satisfies $h^1(N_Y(-\Gamma)) = 0$.
Let $X$ be a general curve of degree $9$ and genus $6$ in $H$ passing through $\Gamma$ (again possible as $X$ satisfies interpolation).  Let $C$ be the reducible curve $X \cup_\Gamma Y$, which is of degree $d= d' + 9$ and genus $g = g' + 12$.  
It suffices to show that $C$ is a BN-curve.

For this, it suffices to show $C$ admits a specialization $C^\circ$
with $h^1(N_{C^\circ}) = 0$ such that $C^\circ$ is a BN-curve;
indeed, $h^1(N_{C^\circ}) = 0$ guarantees the smoothness of the Kontsevich
space at $[C^\circ]$, which implies $[C^\circ]$
lies in a unique component of the Kontsevich space.

We shall degenerate $C$ to $C^\circ = X^\circ \cup_\Gamma Y$
where $X^\circ = X_1 \cup_\Delta X_2$ is a degeneration of $X$
to a reducible curve
which still passes through $\Gamma$;
write $\Gamma = \Gamma_1 \cup \Gamma_2$
with $\Gamma_1 \subset X_1$ and $\Gamma_2 \subset X_2$.
Let $x \in X_1 \smallsetminus \Gamma_1$ be arbitrary.
From the exact sequences
\begin{gather*}
0 \to N_{C^\circ}|_{X_1 \cup_{\Gamma_1} Y} (-\Gamma_2 - \Delta) \to N_{C^\circ} \to N_{C^\circ}|_{X_2} \to 0 \\
0 \to N_{X_1 \cup_{\Gamma_1} Y} (-\Gamma_2 - \Delta - x) \to N_{C^\circ}|_{X_1 \cup_{\Gamma_1} Y} (-\Gamma_2 - \Delta) \to * \to 0 \\
0 \to N_{X_2 / H} \to N_{C^\circ}|_{X_2} \to \O_{X_2}(1)(\Gamma_2) \to 0,
\end{gather*}
where $*$ denotes a punctual sheaf which in particular satisfies $h^1(*) = 0$,
we see that to check $h^1(N_{C^\circ}) = 0$,
it suffices to check $h^1(N_{X_2 / H}) = h^1(\O_{X_2}(1)(\Gamma_2)) = h^1(N_{X_1 \cup_{\Gamma_1} Y} (-\Gamma_2 - \Delta - x)) = 0$.

We shall consider the specialization
where $X_1$ is a rational normal curve,
and $X_2$ is a canonical curve of genus $4$, which meet at a set $\Delta$
of $3$ points. We do this so that $\# \Gamma_1 = 4$ and $\# \Gamma_2 = 3$;
this can be done because
$X_1$ passes through $4$ general points, while $X_2$ passes through
$6$ general points, by \cite[Corollary 1.4]{joint}.
By inspection, $h^1(N_{X_2 / H}) = h^1(\O_{X_2}(1)(\Gamma_2)) = 0$,
so it remains to check
$h^1(N_{X_1 \cup_{\Gamma_1} Y} (-\Gamma_2 - \Delta - x)) = 0$.
This follows from the exact sequences
\begin{gather*}
0 \to N_{X_1 \cup_{\Gamma_1} Y}|_Y(-\Gamma) \to N_{X_1 \cup_{\Gamma_1} Y} (-\Gamma_2 - \Delta - x) \to  N_{X_1 \cup_{\Gamma_1} Y}|_{X_1}(-\Delta - x) \to 0 \\
0 \to N_{X_1 / H}(-\Delta - x) \to N_{X_1 \cup_{\Gamma_1} Y}|_{X_1}(-\Delta - x) \to \O_{X_1}(1)(\Gamma_1 - \Delta - x),
\end{gather*}
where as before $*$ denotes a punctual sheaf which in particular satisfies $h^1(*) = 0$.

It thus remains to show $C^\circ = (X_1 \cup_\Delta X_2) \cup_\Gamma Y = (X_1 \cup_{\Gamma_1} Y) \cup_{(\Gamma_1 \cup \Delta)} X_2$ is a BN-curve.
For this, we can apply Theorem~1.9 of \cite{rbn} twice,
first to conclude that $X_1 \cup_{\Gamma_1} Y$
is a BN-curve, and second to conclude that $C^\circ$ is a BN-curve
as desired --- provided that $X_1 \cup_{\Gamma_1} Y$ admits
a deformation which continues to pass through $\Gamma_1 \cup \Delta$
and is transverse to $H$ along $\Gamma_1 \cup \Delta$.

Since $X_1 \cup_{\Gamma_1} Y$ is already transverse to $H$ along $\Gamma_1$,
this reduces by deformation theory to checking
$h^1(N_{X_1 \cup_{\Gamma_1} Y}(-\Gamma_1 - \Delta - x)) = 0$
for any $x \in \Delta$, which was shown above.
\end{proof}

\begin{cor}\label{gleq14}
If for all $(d,g)$ satisfying $g\leq 14$ and $\rho(d,g,4) \geq 3$, a general BN-curve $Y$ of degree $d$ and genus $g$ satisfies interpolation for the twist $N_Y(-1)$, then a general BN-curve of any genus $g > 14$ does as well.
\end{cor}
\begin{proof}
If $(d,g)$ is BN and $g \geq 15$, then $d \geq 16$.  Hence $d-9 \geq 7$, and we can apply Lemma \ref{dplus9} to inductively reduce to a curve of smaller genus with $\rho \geq 3$.
\end{proof}

Note that all the counterexamples appearing in Theorem~\ref{main}
have $\rho \leq 2$,
so it suffices to prove Theorem~\ref{main} for $g \leq 14$.
We now deal with these remaining cases.  First by attaching twisted cubics at a single point, we show that it suffices to consider three BN-pairs for each $g \leq 14$.

\begin{lem}\label{3ds}
If a BN-pair $(d', g')$ is good, then the pair $(d'+ 3, g')$ is also good.
\end{lem}
\begin{proof}
Attach a general BN-curve $Y$ of degree $d'$ and genus $g'$ at a single point to a general twisted cubic $X \subset H$.  The union has degree $d'+3$ and genus $g'$, and lies in the correct component of $M_{g'}(\pp^4, d'+3)$ by Theorem~1.6 of~\cite{rbn}. 
\end{proof}

\begin{proof}[{Proof of Lemma \ref{cut_to_finite}}]
Corollary \ref{basecases} provides plenty of base cases for the induction: all curves with $d \geq 2g$.    If $g \leq 5$ and $(d,g) \neq (8,5)$ or $(9,5)$, then $d \geq 2g$.  So, using the result of Corollary \ref{gleq14} we will only consider $6\leq g \leq 14$.  For each such $g$, let $d_{\min}$ be the minimal $d$ such that $(d,g)$ is BN and $(d,g) \notin \{(6,2), (8,5), (9,6), (10,7)\}$.
By Lemma \ref{3ds}, it suffices to prove that $(d_{\min}, g), (d_{\min} +1, g),$ and $(d_{\min} + 2, g)$ are good.  We do this using Lemma \ref{in_transverse} with curves $X = (6,3), (7,4)$, and $(8,5)$:
\begin{center}
\begin{tabular}{c | c | c | c | c}
$g$ & $d_{\min}$ & certificate for $(d_{\min}, g)$ & certificate for $(d_{\min}+1, g)$ & certificate for $(d_{\min}+2, g)$ \\ \hline
6 & 10 & $X = (6,3)$, $Y= (4,0)$ & $X = (6,3)$, $Y = (5,0)$ & $12 \geq 2\cdot 6$ \\
7 & 11 & $X = (6,3)$, $Y= (5,1)$ & $X = (6,3)$, $Y = (6,1)$ & $X = (6,3)$, $Y = (7,1)$ \\
8 & 11 & by assumption & $X = (7,4)$, $Y = (5,0)$ & $X = (6,3)$, $Y = (7,2)$ \\
9 & 12 & $X = (7,4)$, $Y= (5,1)$ & $X = (6,3)$, $Y = (7,3)$ & $X = (6,3)$, $Y = (8,3)$ \\
10 & 12 & by assumption & by assumption & $X = (6,3)$, $Y = (8,4)$ \\
11 & 13 & by assumption & $X = (7,4)$, $Y = (7,3)$ & $X = (6,3)$, $Y = (9,5)$ \\
12 & 14 & by assumption & $X = (7,4)$, $Y = (8,4)$ & $X = (6,3)$, $Y = (10,6)$ \\
13 & 15 & $X = (8,5)$, $Y = (7,3)$ & $X = (7,4)$, $Y = (9,5)$ & $X = (6,3)$, $Y = (11,7)$ \\
14 & 16 & $X = (8,5)$, $Y = (8,4)$ & $X = (6,3)$, $Y = (11,8)$ & $X = (6,3)$, $Y = (12,8)$ 
\end{tabular}
\end{center}
In each of these cases, $X \cup_\Gamma Y$ is a BN-curve by Theorem~1.9 of \cite{rbn}.
\end{proof}

In the remainder of the paper we show that the $(d,g)$ pairs 
\[\{(9,5), (11,8), (12,10), (13,10), (13,11), (14,12) \} \] 
are good; and furthermore, we show interpolation for the untwisted normal bundle in the cases $(d,g) \in \{ (8,5), (9,6), (10,7)\}$.


\section{Degeneration with a canonical curve}\label{canonical}

The key argument for $(d,g) \in \{(12,10), (13,10), (13,11), (14,12) \} $ consists of degenerating to the union of a canonical curve $D \subset \pp^4$ of degree $8$ and genus $5$, and another nodal curve.

\textbf{Setup.}
Let $D$ be a general canonical curve ($(d,g) = (8,5)$), and let $p_1, q_1, p_2, q_2, p_3, q_3$ be 6 general points.  Let $\Gamma = p_1 +  q_1 + \cdots + q_3$ denote the sum of these points on $D$.  For each $i$, let $p_i' \in T_{p_i}D$ be a choice of point on the tangent line.  Denote by $L_i$ the line $\bar{p_i, q_i}$.  Suppose that $C \supset L_1 \cup L_2 \cup L_3$ is a connected nodal curve BN-curve containing the lines $L_1, L_2$ and $L_3$.  
Let $Y = C \cup_\Gamma D$ be the nodal curve obtained from gluing $C$ and $D$ along $\Gamma$; this is a BN-curve by Theorem~1.6 of \cite{rbn}.

\begin{lem}\label{canonical_degeneration}
If $N_C(-1)[p_1 \to p_1'][p_2 \to p_2'][p_3 \to p_3']$ satisfies interpolation,
then $N_Y(-1)$ satisfies interpolation as well.
\end{lem}
\begin{proof}
The canonical curve $D$ is the complete intersection of a net of quadrics in $\pp^4$.
Note that no quadric in the net contains all $2$-secant lines to $D$,
as the secant variety is of degree $\binom{8 - 1}{2} - 5 = 16$.
Since it is one condition on a quadric containing $D$ to contain a $2$-secant line to $D$,
we may choose 
three independent quadrics from the net such that each contains exactly two of the three lines $L_1$, $L_2$, and $L_3$.
This gives an isomorphism
\begin{equation}\label{canonical_restriction} N_Y(-1)|_D \simeq K_D(p_1+q_1) \oplus K_D(p_2+q_2) \oplus K_D(p_3+q_3). \end{equation}

Let $w$ be a general point on $D$.  The line bundle $F_i \colonequals K_D(p_i + q_i - w)$ on $D$ is of degree $9$ and has a $5$-dimensional space of global sections, which correspond to meromorphic differentials on $D$ that vanish at $w$ and have at most poles at $p_i$ and $q_i$ with opposite residues.

Since we may compute that $H^0(F_i(-\Gamma)) = 0$, the restriction map
\[H^0(F_i) \xrightarrow{\res} F_i|_{p_1} \oplus F_i|_{q_1} \oplus F_i|_{p_2} \oplus F_i|_{q_2} \oplus F_i|_{p_3} \oplus F_i|_{q_3}, \]
is an isomorphism onto a $5$-dimensional subspace of $F_i |_{\Gamma} \simeq \cc^6$, which can be described
as the preimage of the antidiagonal (i.e.\ the two coordinates negatives
of eachother)
in the $2$-dimensional $F_i|_{p_i} \oplus F_i|_{q_i}$, under the projection map $F_i|_{\Gamma} \to F_i|_{p_i} \oplus F_{i}|_{q_i}$.

The bundle $F_i(-p_1-p_2-p_3) = K_D(q_i - w- p_j-p_k)$ on $D$ satisfies $h^0(F_i(-p_1-p_2-p_3)) = 2$ and $h^1(F_i(-p_1-p_2-p_3)) = 0$, as $h^0(K_D(q_i)) = 5$ and $w$, $p_j$, and $p_k$ are general points on $D$.  For points $(x,y,z)$ in some neighborhood $U \subset D \times D \times D$ of the point $(p_1, p_2, p_3)$, the bundle $F_i(-x-y-z)$ therefore also has $2$ global sections by semicontinuity.  Let $\mathcal{F}_i$ denote the bundle on $D \times U$ which specialized to $F_i(-x-y-z)$ on the fiber $D \times \{(x,y,z)\}$.  Let 
\[\pi_1 \colon D \times U \to D, \qquad \pi_2 \colon D \times U \to U \]
denote the two projections.  

By the theorem on cohomology and base change, the pushforward ${\pi_2}_* \F_i$ is a rank $2$ vector subbundle of a rank $5$ trivial bundle:
\[{\pi_2}_*\F_i \subseteq {\pi_2}_*(\pi_1^*F_i) = H^0(F_i) \otimes\O_U. \]
Composing with the restriction map above yields a map
\[r_i \colon  \ {\pi_2}_* \F_i \to H^0(F_i) \otimes \O_U \to (F_i|_\Gamma \simeq \cc^6) \otimes \O_U, \]
whose fiber over $(x, y, z)$ extracts the image 
\[S^i_{xyz} \colonequals r_i({\pi_2}_* \F_i|_{(x,y,z)}) =  \res(H^0(F_i(-x-y-z))) \subset F_i|_{\Gamma}\] 
for each point $(x,y,z) \in U$.  Write $S_{xyz} \colonequals S^1_{xyz} \oplus S^2_{xyz} \oplus S^3_{xyz}$.

By Lemma \ref{h0_gluing}, in order to prove interpolation for $N_Y(-1)(-x-y-z-w)$, with $x,y,z,w$ general points, it suffices to prove interpolation for the space of sections
\[V_{xyz} \colonequals \left\{\sigma \in H^0(N_Y(-1)|_C) : \sigma|_\Gamma \in S_{xyz} \subseteq N_Y(-1)(-x-y-z-w)|_\Gamma \simeq \cc^{18} \right\}, \]
provided that it has the correct dimension.  Furthermore as satisfying interpolation and having the correct dimension is an open condition, it suffices to prove that 
for the specialization $(p_1, p_2, p_3)$ of the general points $(x,y,z)\in U$, the space of sections
\[V_{p_1p_2p_3} = \left \{\sigma \in H^0(N_Y(-1)|_C) : \sigma|_\Gamma \in S_{x'y'z'}  \right\}, \]
is of the correct dimension and satisfies interpolation.  In this case, 
\[S^i_{p_1p_2p_3} = \res(H^0(K_D(q_i - w -p_j-p_k))) = F_i|_{q_j} \oplus F_i|_{q_k} \subset F_i|_{\Gamma}, \qquad  i,j \neq k.\]
When viewed as meromorphic sections $N_C(-1)$,
the sections of $N_Y(-1)|_C$ whose restrictions to $\Gamma$ lie in the subspace $S_{p_1p_2p_3}$ therefore vanish at the $p_i$ and are regular at the $q_i$.
Therefore, it suffices to prove that
\begin{align*}
V_{p_1p_2p_3} &= \left\{ \sigma \in H^0(N_Y(-1)|_C) : \sigma|_{p_i} = 0, \text{ $\sigma$ is regular at $q_j$} \right\} \\
&=  \left\{ \sigma \in H^0(N_Y(-1)|_C(-p_1-p_2-p_3)) :  \text{ $\sigma$ preserves the node at $q_j$} \right\}
\end{align*}
is of the correct dimension and satisfies interpolation.  This subspace is the space of global sections of the vector bundle
\[N_C(-1)[p_1 \to p_1'][p_2 \to p_2'][p_3 \to p_3'],\]
which has the same rank and Euler characteristic as $N_Y(-1)(-x-y-z-w)$.  Therefore it suffices to prove that this vector bundle satisfies interpolation.
%
\end{proof}


We now deal with each of the cases $(d,g) \in \{(12,10), (13,10), (13,11), (14,12) \} $ separately.

\subsection{Degree 12, genus 10}\label{12_10}

In this case we take $C$ in the setup above to be the union of the three general lines $L_1 \cup L_2 \cup L_3$ and a line $M$ meeting all $3$ lines, which is a BN-curve by Theorem~1.6 of \cite{rbn}.  Let $m_i \colonequals L_i \cap M$ be the points of intersection so that the curve $C$ is precisely the nodal union
\[C\colonequals (L_1 \cup L_2 \cup L_3) \cup_{m_1 + m_2 + m_3} M.\]  
By Lemma \ref{canonical_degeneration}, it suffices to prove interpolation for the modified normal bundle $N_C(-1)[p_1 \to p_1'][p_2\to p_2'][p_3 \to p_3']$. 
By Lemma \ref{1-secant}, we may pull off each of the $1$-secant lines $L_i$ in turn to reduce to proving interpolation for the bundle
\[N_M(-1)(m_1+m_2+m_3)[m_1 \to p_1][m_1 \to p_1+p_1'][m_2 \to p_2][m_2 \to p_2+p_2'][m_3 \to p_3][m_3 \to p_3+p_3'],\]
or, by limiting $p_1'$ to $p_2$, and $p_2'$ to $p_3$, and $p_3'$ to $p_1$
(and using the properties in Proposition~\ref{mod-prop})
for the bundle
\[N_M(-1)(2m_1+2m_2+2m_3)[2m_1\to p_1][m_1 \to p_2][2m_2\to p_2][m_2 \to p_3][2m_3 \to p_3][m_3\to p_1].\]
Limiting $m_2, m_3 \to m_1$ is tree-like, as the points $p_1, p_2$, and $p_3$ are general in $\pp^4$.  Therefore, it suffices to prove interpolation for
\[N_M(-1)(6m_1)[3m_1 \to p_1][3m_1 \to p_2][3m_1\to p_3] \simeq N_M(-1)[3m_1 \to p_1 + p_2 + p_3] \simeq N_M(-1).\]
The result now follows from the fact that $N_M(-1) \simeq \O^{\oplus 3}$
satisfies interpolation.

\subsection{Degree 13, genus 10}\label{13_10}

We continue with the notation of the setup above.
In this case we take $M$ to be a general line meeting $L_1$ and $L_2$ and $N$ to be a general line meeting $L_2$ and $L_3$.  Call $M \cap L_i \equalscolon m_i$ and $N \cap L_j \equalscolon n_j$.  Let 
\[C = (N \cup M) \cup_{m_1+m_2+n_2+n_3} (L_1 \cup L_2 \cup L_3)\]
 be the union, which is a BN-curve by Theorem~1.6 of \cite{rbn}. So by Lemma \ref{canonical_degeneration} it suffices to prove interpolation for $N_C(-1)[p_1 \to p_1'][p_2 \to p_2'][p_3 \to p_3']$.
By Lemma~\ref{1-secant}, this follows from interpolation for the bundle
\[N_1 \colonequals N_{N \cup_{n_2} L_2 \cup_{m_2} M}(-1)(m_1+n_3)[p_2 \to p_2'][m_1 \to p_1][m_1 \to p_1 + p_1'][n_3 \to p_3][n_3 \to p_3 + p_3']. \]
By Lemma~\ref{check_one}, it suffices to show the vanishing of $h^0(N')$
where $N' \colonequals N_1(-x-y)$
where $x$ and $y$ are general points on $N$ and $M$ respectively.
%
%
Restricting to $M$ we have
\begin{align*}
N'|_M & \simeq N_M(-1)(-y+m_1+m_2)[m_1 \to p_1][m_1 \to p_1 + p_1'][m_2 \to p_2] \\
&\simeq N_M[m_1 \to p_1][m_1 \to p_1 + p_1'][m_2 \to p_2]. \end{align*}
As the points $p_1$, $p_2$ and $p_1'$ are independent, we may make a choice of the three hyperplanes $H_i$ defining $M = H_1 \cap H_2 \cap H_3$ so that $p_1 \in H_2 \cap H_3$, $p_2 \in H_1 \cap H_3$ and $p_1' \in H_1 \cap H_2$.  Then the modifications towards these points become simple twists:
\[N'|_M \simeq N_{H_1}|_M(-m_2) \oplus N_{H_2}|_M(-2m_1) \oplus N_{H_3}|_M(-m_1-m_2) \simeq \O_M \oplus \O_M(-1)^{\oplus 2}. \]
This restriction has a unique global section, coming from the factor $N_{H_1}|_M(-m_2)$ of normal directions pointing to $p_1$.  Therefore, applying the argument of Lemma \ref{h0_gluing}, it suffices to show the vanishing of $h^0$
of the subbundle of $N'|_{N \cup_{n_2} L_2}$ whose sections at $m_2$ point towards $p_1$:
\[N_{N \cup_{n_2} L_2}(-1)(-x+n_3+m_2)[p_2 \to p_2'][m_2 \to m_1][m_2 \to p_1][n_3 \to p_3][n_3 \to p_3 + p_3'].\]
We may apply the same argument to $N$ to reduce to showing there are no sections of
\begin{align*}
N_{L_2}(-1)&(n_2+m_2)[p_2 \to p_2'][n_2 \to n_3][n_2 \to p_3][m_2 \to m_1][m_2 \to p_1] \\
& \simeq N_{L_2}(-1)[p_2 \to p_2'][n_2 \to n_3 + p_3][m_2 \to m_1 + p_1].
\end{align*}
Limiting $p_2$ to $n_2$, which is tree-like as $p_2'$ is general, it suffices to show the vanishing of $h^0$ for
\begin{align*}
N_{L_2}(-1)(-n_2)[n_2 \to n_3 + p_3 + p_2'][m_2 \to m_1 + p_1] &\simeq N_{L_2}(-1)(-n_2)[m_2 \to m_1 + p_1],
\intertext{and as $m_1$ and $p_1$ are in general directions, we have}
& \simeq \O_M(-2) \oplus \O_M(-1) \oplus \O_M(-1),
\end{align*}
which has no global sections, as desired.

\subsection{Degree 13, genus 11}\label{13_11}

As in the setup above, we choose $C$ to contain the lines $L_1, L_2, L_3$ as well as two further lines $M$ and $N$, where $M$ meets each $L_i$ at a single point $m_i$ and $N$ meets $L_1$ and $L_2$ at points $n_1$ and $n_2$ respectively.  As the lines $L_1$, $L_2$, and $L_3$ are general, the lines can be assumed to satisfy various genericity conditions: $N \cup M \cup L_i$ span a $\pp^3$ for $i=1,2$. 

Then the resulting union 
\[C \colonequals (N \cup M) \cup_{m_1+m_2+m_3+n_1+n_2} (L_1 \cup L_2 \cup L_3) \]
is a nodal connected curve of degree $5$ and arithmetic genus $2$, which is a BN-curve by Theorem~1.6 of \cite{rbn}.  By Lemma \ref{canonical_degeneration}, it suffices to show that 
\[N_C' \colonequals N_C(-1)[p_1 \to p_1'][p_2 \to p_2'][p_3 \to p_3']\]
satisfies interpolation.  By Lemma \ref{2-secant-easy}, it suffices to show that 
\[N_{(N \cup M) \cup_{m_2+m_3+n_2} (L_2 \cup L_3)}(-1)[p_2 \to p_2'][p_3 \to p_3'][n_1 \to m_1][m_1 \to n_1] \]
satisfies interpolation.  By Lemma \ref{1-secant}, this reduces in turn
to showing that
\[N_{M \cup_{m_2+m_3} L_2 \cup L_3}(-1)(n_2)[p_2 \to p_2'][p_3 \to p_3'][m_1 \to n_1][n_2 \to n_1][n_2 \to n_1+ m_1] \]
satisfies interpolation.  Again by Lemma \ref{1-secant},
this further reduces to proving interpolation for
\[ N_{M \cup_{m_2} L_2}(-1)(n_2 + m_3)[p_2 \to p_2'][m_1 \to n_1][n_2 \to n_1][n_2 \to n_1 + m_1][m_3 \to p_3][m_3 \to p_3 + p_3']. \]
We may limit $p_2 \to n_2$, since $p_2'$ is disjoint from $\bar{n_1, m_1}$ and so this is tree-like.  We are therefore reduced to proving interpolation for the vector bundle
\begin{align*}
N_{M\cup_{m_2} L_2}&(-1)(n_2 + m_3)[n_2 \to p_2'][m_1 \to n_1][n_2 \to n_1][n_2 \to n_1 + m_1][m_3 \to p_3][m_3 \to p_3 + p_3'] \\
&\simeq N_{M\cup_{m_2} L_2}(-1)(m_3)[m_1 \to n_1][n_2 \to n_1][n_2 \to n_1 + m_1+p_2'][m_3 \to p_3][m_3 \to p_3 + p_3']
\intertext{As $p_2'$ is general and $m_1$, $n_1$, and $L_2$ span a $\pp^3$, the linear space $\bar{p_2', m_1, n_1, L_2} = \pp^4$.  Therefore the above bundle is isomorphic to}
&\simeq N_{M\cup_{m_2} L_2}(-1)(m_3)[m_1 \to n_1][n_2 \to n_1][m_3 \to p_3][m_3 \to p_3 + p_3'].
\end{align*}
We may now apply Lemma \ref{1-secant} a final time to reduce to showing interpolation for
 \[N_M(-1)(m_2 + m_3)[m_2 \to n_2][m_2 \to n_2 + n_1][m_1 \to n_1][m_3 \to p_3][m_3 \to p_3 + p_3']. \]
Limit $p_3'$ to $n_1$.  Then we may limit $m_2$ and $m_3$ to $m_1$, as $n_1, n_2$ and $p_3$ are independent so this is tree-like.  The resulting vector bundle is
\begin{align*}
N_M(-1)&(2m_1)[m_1 \to n_2][m_1 \to n_2 + n_1][m_1 \to n_1][m_1 \to p_3][m_1 \to p_3 + n_1] \\
&\simeq N_M(-1)(4m_1)[m_1 \to n_1][2m_1 \to n_1][2m_1 \to n_2][2m_1\to p_3]\\
& \simeq N_M(-1)[m_1 \to n_1][2m_1 \to n_1 + m_1 + p_3],
\intertext{and as $n_1$, $m_1$ and $p_3$ are independent, and so span all of $\pp^4$, we have}
& \simeq N_M(-1)[m_1 \to n_1]\\
&\simeq \O_M \oplus \O_M(-1)^{\oplus 2},
\end{align*}
which satisfies interpolation by inspection.

\subsection{Degree 14, genus 12}\label{14_12}

Continuing in the setup above, let $M$ be a general line meeting each $L_i$ at $m_i$ and let $N$ be a general conic meeting each $L_i$ at $n_i$.
Pick points $n_i' \neq n_i$ on $T_{n_i} N$,
and a point $m \in M$ distinct from $m_1$, $m_2$, and $m_3$.
Let $C$ denote the union 
\[C \colonequals (M \cup N) \cup_{m_1+m_2+m_3+n_1+n_2+n_3} (L_1 \cup L_2 \cup L_3),\]
 which is a BN-curve by Theorem~1.6 of \cite{rbn}.  It then suffices to prove interpolation for $N_C' \colonequals N_C(-1)[p_1 \to p_1'][p_2\to p_2'][p_3 \to p_3']$.  This vector bundle has Euler characteristic $5$.  So by Lemma \ref{check_one}, it suffices to show that for $2$ points $x + y$, the twist down satisfies 
\begin{equation}\label{14_12_vanishing}h^0(N_C(-1)(-x-y)[p_1 \to p_1'][p_2 \to p_2'][p_3 \to p_3'])=0.\end{equation}
We will choose $x$ and $y$ to be general points on $M$, so that $N_C'(-x-y)|_M \simeq \O(-1)^{\oplus 3}$.  Therefore to show that the vanishing of \eqref{14_12_vanishing}, it suffices to show that 
\[h^0(N_C'|_{N \cup_{n_1+n_2+n_3} L_1 \cup L_2 \cup L_3}(-m_1-m_2-m_3)) = 0, \]
as in the proof of Lemma \ref{3-secant-twist}.  This bundle restricts to each $L_i$ to be
\[ N_{L_i}(-1)(n_i)[n_i \to n_i'][p_i \to p_i'][m_i \to m] \simeq \O_{L_i}(-1)^{\oplus 3}, \]
as $L_i$, $n_i'$, $p_i'$, and $m$ are linearly independent.
Therefore as in Lemma \ref{h0_gluing}, it suffices to show that
\[h^0(N_N(-1)[n_1 \to p_1][n_2\to p_2][n_3 \to p_3]) = 0, \]
where the $p_i$ are in general directions.  By the semicontinuity theorem it suffices to prove this vanishing after specializing $p_1$, $p_2$ and $p_3$.  Fix a quadric $Q$ and hyperplanes $H_1$ and $H_1$ such that $N$ is the complete intersection $Q\cap H_1 \cap H_2$.  We have 
\begin{equation}\label{decomp} N_N(-1) \simeq N_Q(-1)|_N \oplus N_{H_1}(-1)|_N \oplus N_{H_2}(-1)|_N \simeq \O_N(1) \oplus \O_N^{\oplus 2}. \end{equation}
Specialize $p_1$ and $p_2$ so that they are general in
$T_{n_1}(Q\cap H_2)$ and $T_{n_2}(Q \cap H_2)$ respectively,
and $p_3$ so that it is general in $T_{n_3}(Q\cap H_1)$.  Then the negative modifications are towards specific factors in the decomposition \eqref{decomp}, and so become simple twists:
\begin{align*}
N_N(-1)[n_1 \to p_1][n_2\to p_2][n_3 \to p_3] &\simeq \O_N(1)(-p_1-p_2-p_3) \oplus \O_N(-p_3) \oplus \O_N(-p_1-p_2),
\intertext{and under an isomorphism $M \simeq \pp^1$,}
& \simeq \O_{\pp^1}(-1) \oplus \O_{\pp^1}(-1) \oplus \O_{\pp^1}(-2),
\end{align*}
which has no global sections, as desired.



\section{Degeneration with a rational quartic curve}\label{quartic}

Let $X$ be a nondegenerate curve of degree $d$ and genus $g$ in $\pp^4$.  Let $(p_1, q_1), (p_2, q_2), (p_3, q_3)$ be three pairs of general points on $X$.  Let $\Gamma \colonequals p_1 + p_2 + p_3 + q_1 + q_2 + q_3$.  Denote by $L_i$ the 2-secant line through $p_i$ and $q_i$.  There is a unique fourth line $M$ meeting each of the $L_i$ (by, for example, intersection theory
in the Grassmannian of lines in $\pp^4$).
Call $L_i \cap M = m_i$.  As $X$ is nondegenerate, the union $L_1 \cup L_2 \cup L_3 \cup M$ spans $\pp^4$.  Call this union 
\[Y \colonequals X \cup_\Gamma \big((L_1 +L_2 +L_3) \cup_{m_1 + m_2 + m_3} M\big),\]
 which is a BN-curve by Theorem~1.6 of~\cite{rbn}.
We will use this particular degeneration to prove interpolation for $N_Y(-1)$ when $Y$ is a curve of degree $11$ and genus $8$, and for $N_Y$ when $Y$ is a curve of degree $9$ and genus $6$ (respectively degree $10$ and genus $7$).


\begin{lem}\label{3_secant}
With notation as above, 
\begin{enumerate}[(i)]
\item $N_Y(-1)$ satisfies interpolation if 
\[N_Y(-1)|_X(-p_1-q_1-p_2-q_2 -p_3-q_3) \simeq N_X(-1)[p_1 \to q_1][q_1 \to p_1][p_2 \to q_2][q_2 \to p_2][p_3 \to q_3][q_3 \to p_3] \]
satisfies interpolation.
\item $N_Y$ satisfies interpolation if 
\[ N_Y|_X = N_X(p_1+q_1+p_2+q_2 +p_3+q_3)[p_1 \to q_1][q_1 \to p_1][p_2 \to q_2][q_2 \to p_2][p_3 \to q_3][q_3 \to p_3] \]
satisfies interpolation.  This vector bundle satisfies interpolation if the modified normal bundle
\[N_X' \colonequals N_X[q_1 + q_2 + q_3 \to p_1 ] \]
satisfies interpolation.
\end{enumerate}
\end{lem}
\begin{proof}
As $L_1 \cup L_2 \cup L_3 \cup M$ spans $\pp^4$, the three directions of the $L_i$ at the points $m_i$ are independent.  
By Lemma \ref{3-secant-twist} (respectively Lemma \ref{3-secant}), it suffices to prove interpolation for $N_Y(-1)|_{X \cup_\Gamma (L_1 \cup L_2 \cup L_3)}(-m_1-m_2-m_3)$ (respectively $N_Y|_{X \cup_\Gamma (L_1 \cup L_2 \cup L_3)}$) in order to deduce interpolation for $N_Y(-1)$ (respectively $N_Y$).  

For each of the $L_i$, the three lines $T_{p_i}X$, $T_{q_i}X$ and $M$ span $\pp^4$,
and so we similarly have that $N_Y|_{L_i} \simeq \O_{L_i}(2)^{\oplus 3}$.  To prove interpolation for $N_Y(-1)$ and $N_Y$, we are free to pick our general points to include a choice of general point $x_i$ on each of the $L_i$.  Twisting down by this point for each $L_i$, we have
\[N_Y|_{L_i}(-x_i) \simeq \O_{L_i}(1)^{\oplus 3}, \qquad N_Y(-1)|_{L_i}(-m_i-x_i) \simeq \O_{L_i}(-1)^{\oplus 3}. \]
To prove interpolation for $N_Y(-1)$, by Lemma \ref{2-secant-easy}, we may peel off each line $L_i$ in turn to reduce to proving interpolation for
\[N_Y(-1)|_X(-p_1-q_1-p_2-q_2 -p_3-q_3). \]
To prove interpolation for $N_Y$, note that evaluation at the two points $p_i, q_i \in L_i \simeq \pp^1$ is an isomorphism
\[H^0(\O_{\pp^1}(1))^{\oplus 3} \xrightarrow{\ev} \O_{\pp^1}(1)^{\oplus 3}|_{p_i + q_i} .\]
Hence sections of $N_Y(-x_1-x_2-x_3)$ on $X$ extend uniquely to $X \cap L_1\cap L_2 \cap L_3$ and then to all of $Y$.   Therefore by Corollary \ref{h0_gluing_vb}, interpolation for $N_Y(-x_1-x_2-x_3)$ (and hence for $N_Y$) follows from interpolation for $N_Y|_X$.  The formula for $N_Y|_X$ follows directly from Lemma \ref{normal_bundle_nodal}.

Now limit $p_2$ and $p_3$ to $p_1$; the corresponding modification is tree-like as $q_1$, $q_2$,
and $q_3$ are independent.  By the properties of Proposition \ref{mod-prop}, the resulting bundle is therefore
\begin{align*}
N_X&(3p_1+q_1+q_2 +q_3)[p_1 \to q_1][q_1 \to p_1][p_1 \to q_2][q_2 \to p_1][p_1 \to q_3][q_3 \to p_1] \\
&\simeq N_X(p_1+q_1+q_2 +q_3)[p_1 \to q_1 + q_2 +q_3][q_1+q_2+q_3 \to p_1] \\
& \simeq  N_X(p_1+q_1+q_2 +q_3)[q_1+q_2+q_3 \to p_1].
\end{align*}
Interpolation for this bundle therefore follows from interpolation for $N_X[q_1+q_2+q_3 \to p_1]$.
\end{proof}

\subsection{Degree 11, genus 8}\label{11_8}

By Lemma \ref{3_secant}, we may reduce interpolation for the twisted normal bundle of a general BN-curve of degree $11$ and genus $8$ to interpolation for the modified normal bundle $N_X(-1)[p_1 \to q_1][q_1 \to p_1][p_2 \to q_2][q_2 \to p_2][p_3 \to q_3][q_3 \to p_3] $, where $X$ is a BN-curve of degree $7$ and genus $3$.  

We will further degenerate using the following.  As before, $N_X'$ denotes some modification of $N_X$, which is isomorphic to $N_X$ over some open containing all of $L$.  $N_C'$ is then the bundle on $C$ obtained by gluing $N_X'|_{C\smallsetminus C\cap L}$ to $N_C$ along this open.

\begin{lem}\label{mod_on_secant}
Let $C$ be a BN-curve.  Let $L$ be a $2$-secant line meeting $C$ at two general points $x$ and $y$; write $X $ for the union $X \colonequals C\cup_{x+y} L$.  Suppose that $p_1$ and $p_2$ are general points on $L$ and $q_1 and 1_2$ are general points on $C$, and $D_1$ and $D_2$ are divisors on $C$.
If the bundle
\[N_C'(-1)(y)[x \to y][2y \to x][D_1 \to x][D_2 \to x] \]
on $C$ satisfies interpolation, then the bundle
\begin{equation}\label{starting-bundle}N_X'' \colonequals N_X'(-1)[p_1 \to q_1][D_1 \to p_1][p_2 \to q_2][D_2 \to p_2]\end{equation}
on $X$ satisfies interpolation.
\end{lem}
\begin{proof}
Let $x' \in T_xC$ and $y' \in T_y C$ be choices of points on the tangent lines
distinct from $x$ and $y$ respectively.
If we restrict the bundle $N_X'$ of \eqref{starting-bundle} to $L$, we have
\[N_X''|_L \simeq N_L(-1)(x+y)[x\to x'][y \to y'][p_1 \to q_1][p_2 \to q_2] \simeq \O_L \oplus \O_L(-1)^{\oplus 2},  \]
which has one global section.  We will do a calculation in local coordinates to more precisely understand the behavior of this section at $x$ and $y$ as a meromorphic section of $N_L(-1)$.
  
The twisted normal bundle $N_L(-1)\simeq V \otimes \O_L$ is trivial of rank $3$ on $L$ (with $V \simeq H^0(N_L(-1)$).  We may therefore assign, to points of $\pp^4 \smallsetminus L$, the corresponding normal direction to $L$ in $\pp V$.
Choose appropriate coordinates for $V$ so that
\[x' \mapsto [1:0:0], \ y' \mapsto [0:1:0], \ q_1 \mapsto [1:0:1], \ q_2 \mapsto [0:1:1].\]
In addition let $t$ be an affine coordinate on $L \simeq \pp^1$ so that $x$ corresponds to $t=0$, $y$ corresponds to $t=1$, and $p_1$ and $p_2$ correspond to $t=a$ and $t=b$ respectively.  Then the unique section of $N_X'|_L$ up to scaling is given by (scalar multiples of) the tuple 
\[ \sigma_t = \left[ \frac{a(t-b)}{t} , \frac{(1 - b)(t-a)}{t-1} , a - b \right] \] 
as a meromorphic function valued in $ V$. 
The image of this section under the restriction map to $x$ (i.e. $t=0$) is
\[\sigma_0 = \left[-ab , a(1 - b) , a - b \right]\]
and under restriction to $y$ (i.e. $t=1$) is
\[\sigma_1 = [a(1-b) , (1-a)(1-b) , a-b]. \]
By Lemma \ref{h0_gluing}, it suffices to prove interpolation for the space of sections of $N_X''|_C$ whose values at $x$ and $y$ are in the $1$-dimensional subspace given by $(\sigma_0,\sigma_1)$ in $ N_{C\cup_{x+y} L}'(-1)|_{x + y}$.  As interpolation is open, we may limit $p_1$ and $p_2$ to $x$, in which case
$\sigma_0$ and $\sigma_1$ limit to
\[ \sigma_0' = [ 0 , 0 , 0], \qquad \sigma_1' = [0,1,0]. \]
Therefore the sections of $N_X''|_L$ must vanish at $x$ and point towards $y'$ at $y$.  By \cite[Lemma 8.4]{joint}, the bundle $N_X''|_L[y \to x]$ glues to $N_X''|_L[y \to y']$ at the point $y$.
Therefore it suffices to prove interpolation for the vector bundle
\begin{align*}
N_{C \cup_{x+y} L}'&(-1)[D_1 \to x][D_2 \to x]|_C(-x)[y \to x] \\
&\simeq N_C'(-1)(y)[x \to y][2y \to x][D_1 \to x][D_2 \to x]
\end{align*}
on $C$.
\end{proof}

By Theorem~1.6 of \cite{rbn}, we may degenerate our BN-curve of degree $7$
and genus $3$ to the union of 
a BN-curve of degree $6$ and genus $2$ and a $2$-secant line $L$ meeting $C$ at points $x$ and $y$.  Do this so the points $p_1$ and $p_2$ specialize on to $L$ and the remainder of the marked points specialize to $C$.  Then we apply Lemma \ref{mod_on_secant} with $N_X' = N_X[p_3 \to q_3][q_3 \to p_3]$ and $D_1=q_1$ and $D_2= q_2$.  We thus reduce to proving interpolation for the bundle
\[N_C(-1)(y)[x \to y][2y \to x][q_1 \to x][q_2 \to x][p_3 \to q_3][q_3 \to p_3].\]

By Theorem~1.6 of \cite{rbn}, we may further degenerate $C$
to the union of a curve $D$ of degree $5$ and genus $1$ and a $2$-secant line $M$, meeting $D$ at points $z$ and $w$.  Let the points $p_3$ and $q_1$ specialize to general points on $L$, and $x$, $q_3$, $q_2$, and $y$ specialize to general points on $D$.  Again we apply Lemma \ref{mod_on_secant} with $p_1 = p_3$, $p_2 = q_1$, $q_1 = q_3$, $q_2 = x$, $D_1 = q_2$ and $D_2 = \emptyset$.  It therefore suffices to prove interpolation for the bundle
\[N_D(-1)(y+w)[x\to y][2y \to x][q_2 \to x][q_3 \to z][2w \to z][z \to w] \]
on $D$.  We may limit $w$ to $x$, to obtain the bundle
\[N_D' \colonequals N_D(-1)(x+y)[x\to y][2y \to x][q_2 \to x][q_3 \to z][2x \to z][z \to x], \]
which has Euler characteristic $0$.  This bundle sits in a balanced exact sequence
\[0 \to N_{D \to x}(-1)(y-q_3 -2x) \to N_D' \to \pi_x|_{D}^*N_{\pi_x|_{D}}(-1)(x-y-q_2-z)[x\to y][q_3 \to z][2x \to z] \to 0, \]
Therefore by \cite[Proposition 4.16]{joint} and the fact that as a nonspecial line bundle $N_{D \to p}(-1)(y-q_3-2x)$ satisfies interpolation \cite[Propsition 4.7]{joint}, it suffices to prove that 
\begin{align*}
\pi_x|_{D}^*N_{\pi_x|_{D}}(-1)(x-y-q_2-z)[x\to y][q_3 \to z][2x \to z] & \simeq \pi_x|_{D}^*N_{\pi_x|_{D}}(-1)(-y-q_2-z)[q_3 + x \to z][x \to y + z] \\ &\simeq \pi_x|_{D}^*N_{\pi_x|_{D}}(-1)(-y-q_2-z)[q_3 + x \to z]
\end{align*}
satisfies interpolation as a vector bundle on $D$.  As $\pi_x|_{D}^*\O(1)(y + q_2 + z)$ is a general line bundle on $D$, it suffices to show that $\pi_x|_{D}^*N_{\pi_x|_{D}}[q_3 + x \to z]$ satisfies interpolation.
This is verified by publicly available code in \cite[Appendix B]{joint}
to prove interpolation for
certain modifications of normal bundles of nonspecial curves:

\begin{verbatim}
>>> good(Curve(4,1,3).add(P100,2))
True
\end{verbatim}


\subsection{Degree 9, genus 6}\label{9_6}

By Lemma \ref{3_secant}, interpolation for the normal bundle of a general curve of degree $9$ and genus $6$ follows from interpolation for the modified normal bundle
\[N_C' \colonequals N_C[q_1+q_2+q_3 \to p]\]
of a curve $C$ of degree $5$ and genus $1$ in $\pp^4$.
This is verified by publicly available code in \cite[Appendix B]{joint}:

\begin{verbatim}
>>> good(Curve(5,1,4).add(P100,3))
True
\end{verbatim}

%

\subsection{Degree 10, genus 7} \label{10_7}

As above, by \ref{3_secant}, it suffices to show that for a BN-curve $X \to \pp^4$ of degree $6$ and genus $2$, and $p, q_1, q_2, q_3$ general points of $X$, the modified normal bundle
\[N_X' = N_X[q_1+q_2+q_3 \to p] \]
satisfies interpolation.  This is verified by publicly available code in \cite[Appendix B]{joint}:

\begin{verbatim}
>>> good(Curve(6,2,4).add(P100,3))
True
\end{verbatim}

\section{Proofs of the Main Theorems}\label{proofs}

In this section we wrap up the preceding computations to prove the theorems and corollaries quoted in the introduction.

\subsection{Proof of Theorem~\ref{main}}

As interpolation is an open condition \cite[Theorem 5.8]{nasko}, it suffices to find one BN-curve $C$ of every possible degree $d$ and genus $g$ (with $\rho(d,g,4) \geq 0$) whose twisted normal bundle $N_C(-1)$ satisfies interpolation.  By Proposition \ref{cut_to_finite}, this follows for all BN-pairs $(d,g) \notin \{(6,2), (8,5), (9,6), (10,7)\}$ from checking interpolation for BN-curves with $d$ and $g$ in the following finite list
\[\{(9,5), (11,8), (12,10), (13,10), (13,11), (14,12)\}. \]

For $(d,g) \in \{(12,10), (13,10), (13,11), (14,12)\}$, this was proved in Sections \ref{12_10} -- \ref{14_12} using a degeneration containing a canonical curve.  For $(d,g) = (11,8)$, this was proved in Section \ref{11_8} using a degeneration with a rational quartic curve as well as several 2-secant degenerations.  Let us show that the twisted normal bundle $N_C(-1)$ of a general BN-curve $C$ of degree $9$ and genus $5$ satisfies interpolation.  As $\chi(N_C(-1)) = 14 \equiv -1 \pmod{3}$, it suffices by Lemma \ref{check_one} to show that for $5$ general points $p_1, \ldots, p_5$, the twist down satisfies
\[h^0(N_C(-1)(-p_1- \cdots -p_5)) = 0. \]
As the genus of $C$ is $5$, the line bundle $\O_C(1)(p_1 +\cdots +p_5)$ is general of degree $14$.  Therefore this vanishing follows from interpolation for the untwisted normal bundle $N_C$, which is a special case of
\cite[Theorem 1.3]{joint}.

For the converse, we note that $h^1(N_C(-1)) \neq 0$
if $C \to \pp^4$ is a BN-curve of degree and genus in $\{(8, 5), (9, 6), (10, 7)\}$
by \cite[Theorem 1.6]{quadrics}.
Moreover, $N_C$ does not satisfy interpolation
if $C \to \pp^4$ is a BN-curve of degree $6$ and genus $2$.
\qed

\subsection{Proof of Theorem~\ref{notwist}}

For any BN-curve $C$, if the twisted normal bundle $N_C(-1)$ satisfies interpolation, then the untwisted normal bundle $N_C$ does as well \cite[Proposition 4.11]{joint}.  Therefore for all BN-pairs $(d,g) \notin \{(6,2), (8,5), (9,6), (10,7)\}$, this is a corollary of Theorem~\ref{main}.  If $C$ is of degree $6$ and genus $2$, then $C$ is nonspecial and $N_C$ was already observed to fail to satisfy interpolation in \cite{joint}.  In Sections \ref{9_6} and \ref{10_7} we proved interpolation for $(d,g) = (9,6)$ (respectively $(d,g) =(10,7)$). 

A BN-curve $C$ of degree $8$ and genus $5$ in $\pp^4$ is a canonical curve,
which is a complete intersection of $3$ quadrics; its normal bundle is therefore
\[N_C \simeq (K_C^{\otimes 2})^{\oplus 3}. \]
Interpolation for $N_C$ thus reduces to interpolation
for
$K_C^{\otimes 2}$, which holds since $K_C^{\otimes 2}$ is nonspecial \cite[Proposition 4.7]{joint}.
\qed

\subsection{Proof of Corollaries~\ref{cor_twist} and~\ref{cor_notwist} }

When $\rho(d,g,4) \geq 0$, write $\bar{M}_{g,n}(\pp^4, d)^\circ$ for the unique component dominating $\bar{M}_g$ of the Kontsevich space of $n$-marked genus $g$ curves of degree $d$ in $\pp^4$.  If $n \leq f(d,g,4)$, and $p_1, \ldots, p_n$ are general points on $C$, then $\chi(N_C(-p_1-\cdots-p_n)) \geq 0$ and interpolation for $N_C$ implies that $h^1(N_C(-p_1-\cdots-p_n)) = 0$.  Therefore the map extracting the image of the $n$ marked points
\[\bar{M}_{g,n}(\pp^4, d)^\circ \to (\pp^4)^n \]
is smooth at the point $(f \colon C \to \pp^4, p_1, \ldots, p_n)$, and hence generically smooth.  Therefore the map is dominant.  For degree $6$ and genus $2$, even though the normal bundle does not satisfy interpolation; however it is still true that for $9 = f(6,2,4)$ general points $p_1, \ldots, p_9$ on $C$, the twist $h^1(N_C(-p_1-\cdots -p_9))=0$.  Therefore, the above evaluation map is still generically smooth and a general curve of this degree and genus still passes through $9$ general points \cite[Corollary 1.4]{joint}.

Similarly, fix a hyperplane $H \subset \pp^4$; then interpolation for $N_C(-1)$ implies that the map extracting the section by $H$ (when it is transverse) and the image of the $n$ points
\[ \bar{M}_{g,n}(\pp^4, d)^\circ \dashrightarrow H^d \times (\pp^4)^n \]
is smooth, and hence dominant, for $n \leq f(d,g,4) - d$.  In the case $(d,g)=(6,2)$, if $p_1, p_2, p_3$ are general points on $C$, then as $3 > 2$, we have that $\O_C(1)(p_1 + p_2 + p_3)$ is a general line bundle of degree $9 = f(6,2,4)$.  Therefore the fact that $C$ passes through $9$ general points implies that its hyperplane section is general and it passes through $3$ independently general points.

We conclude with the three cases $(d,g) \in \{(8,5), (9,6), (10,7)\}$.  In these cases $h^1(N_C(-1)) \neq 0$, and hence the hyperplane section is not a general collection of $d$ points in $\pp^3$; however $d-1$ of these points are general, and all $d$ points are general subject to the condition that they are distinct points
on the complete intersection of $11 - d$ quadrics
\cite[Theorem 1.6]{quadrics}.

\subsection{Degree 8, genus 5}

Let $D = p_1 + \cdots +p_{7}$ be a collection of $d-1$ points in a general hyperplane section, then we will show that $N_C(-D)$ satisfies interpolation.  We have, 
\[\chi(N_C(-D)) = 15 \geq \genus(C)\cdot \rk(N_C).\]
Therefore interpolation follows from interpolation for $N_C$ as in \cite[Proposition 4.12]{joint}.  As above this suffices to show that a general curve of degree $8$ and genus $5$ passes through $f(8,5,4) = 12$ points that are general subject to the constrain that $7$ of them lie in a hyperplane.

In fact, something a priori stronger follows.  The hyperplane section of a curve of degree $8$ and genus $5$ (which is the complete intersection of $3$ quadrics in $\pp^4$) is the complete intersection of $3$ quadrics in $\pp^3$;
thus, some subset $D$ of $7$ of them determines the $8$th.  Therefore any canonical curve in $\pp^4$ passing through these $7$ general points $D$ in a hyperplane passes through the $8$th point automatically.  In this way, the previous result implies that one may pass a curve of degree $8$ and genus $5$ through $f(d,g,4)+1$ points that are general subject to the constraint that $d$ lie in a hyperplane and
are distinct points on the complete intersection of $11 - d$ quadrics.

\subsection{Degree 9, genus 6}

We first show that for $D = p_1 + \cdots +p_{8}$ a collection of $d-1$ general points in a hyperplane, $N_C(-D)$ satisfies the property of interpolation.  As above this implies the corollary about passing a general such curve through $f(d,g,4)$ points subject to the constraint that $d-1$ lie in a hyperplane.
When $(d,g) = (9,6)$, we have $\chi(N_C(-D)) = 16$.   It suffices, therefore, (by Lemma~\ref{check_one}) to show that for a general divisor $\Gamma$ of degree $5$, the twist down has no higher cohomology:
\[h^1(N_C(-\Gamma-D)) = 0.\]
As $D = \O_C(1)(-p)$, with $p$ general on $C$, the above twist $N_C(-\Gamma-D) \simeq N_C(-1)(p-D)$.  If $\Gamma$ is a general collection of $5$ points, then $\O_C(p-D)$, and hence $\O_C(-1)(p-D)$, is a general line bundle.  The result therefore follows from Theorem~\ref{notwist}.

In a slightly different direction, we show that a curve of degree $9$ and genus $6$ in $\pp^4$ passes through $f(9,6,4)= 13$ points that are general subject to the constrain that $d=9$ of them lie in a hyperplane and lie on the complete
intersection of $2$ quadrics,
i.e.\ an elliptic normal curve.  Note that an elliptic normal curve in $\pp^3$ passes through $8$ general points and is determined by $8$ general points on the curve.

Fix $9$ general points $p_1, \ldots, p_9$ on an elliptic normal curve $E \subset H \simeq \pp^3$ and $5$ other general points $q_1, \ldots, q_5$ in $\pp^4$.  We will show that there exists a curve of degree $9$ and genus $6$ through the first $13$ points $p_1 + \cdots + p_9 + q_1 + \cdots + q_4$ but not through all $14 = f(9,6,4)$ points.

Let $L$ be a line in $\pp^4$ which is general relative to $p_1 , \ldots, p_7$ and $q_1, q_2, q_3$
(and so in particular disjoint from $p_1, \ldots, p_7, q_1, q_2, q_3$).  

\begin{lem}\label{sm_dP}
The subscheme 
\[q_1 \cup q_2 \cup q_3 \cup p_1 \cup \cdots \cup p_7 \cup L \subset \pp^4\]
lies on a smooth del Pezzo surface of degree $4$.  
\end{lem}
\begin{proof}
We will show that if $S$ is some del Pezzo surface, $L$ is a line on $S$, and $q_1, q_2, q_3$ are general points on $S$ and $p_1, \ldots, p_7$ are general points in a hyperplane section of $S$, then
\[H^0(N_S) \to H^0(N_S|_{q_1 \cup q_2 \cup q_3 \cup p_1 \cup \cdots \cup p_7 \cup L})\]
is surjective, and $h^1(N_S) = 0$, so that these points and line are general by deformation theory.  As $N_S \simeq \O_S(2)^{\oplus 2}$, we see from Kodaira vanishing that $h^1(N_S) = 0$.
As the points $q_i$ are general on $S$ and so impose independent conditions on sections of a line bundle,  it suffices to show that
\[H^0(\O_{S}(2)) \to H^0(\O(2)|_{ p_1 \cup \cdots \cup p_7 \cup L})\]
is surjective, for which it suffices to show the composition
\[H^0(\O_{\pp^4}(2)) \to H^0(\O_{\pp^3 \cup L}(2)) \to H^0(\O(2)|_{ p_1 \cup \cdots \cup p_7}) \oplus H^0(\O_L(2)) \]
is surjective.  Consider the second map first; the map evidently surjects onto $H^0(\O_L(2))$, so it suffices to check that the kernel surjects onto $H^0(\O_{p_1 \cup \cdots \cup p_7}(2))$.
Write $p = L \cap H$. Note that any quadric on $H$ vanishing at $p$ can be extended to a quadric on $\pp^4$
vanishing on $L$ (for example by pullback under the projection map from some point on $L \smallsetminus \{p\}$).
The restriction of the kernel to $H$ is therefore isomorphic to $H^0(\O_{\pp^3}(2) \otimes \I_p)$, where $\I_p$ is the ideal sheaf of the point $p$, which surjects onto $H^0(\O_{p_1 \cup \cdots \cup p_7}(2))$ as the $8$ points $p_1, \ldots, p_7, p$ are general in $\pp^3$.
\end{proof}

To complete the proof, define a point $p$ on $E$ by the relation
\[\O_E(p_1 + \cdots + p_9)(-2) = \O_E(p).\]
Note that $p$ is general in $H$ independent from $p_1, \ldots, p_7, q_1, q_2, q_3, q_4$.
Let $L$ be the line in $\pp^4$ joining $p$ and $q_4$.  Then by Lemma~\ref{sm_dP} above, there exists a unique smooth del Pezzo surface $S$ through $q_1 \cup q_2 \cup q_3 \cup p_1 \cup \cdots \cup p_7 \cup L$.  The hyperplane section $S \cap H$ is an elliptic normal curve containing the $8$ points $p_1, \ldots, p_7$ and $p$ on $E$, and hence must be equal to $E$.  It therefore contains the $13$ points $p_1, \ldots, p_9$ and $q_1, \ldots, q_4$.

The curve class $2H + L$ is base-point-free of degree $9$ and arithmetic genus $g$ satisfying
\[2g - 2 = (H + L)\cdot (2H+L) = 8 + 3 - 1 =10 \qquad \Rightarrow \qquad g = 6. \]
The general member of this linear system on $S$ is therefore a smooth (hence BN by \cite{iliev}) curve of degree $9$ and genus $6$ in $\pp^4$.  It suffices only to show that such a curve can be passed through $4$ general points on $S$ and $8$ general points in a hyperplane section.  By deformation theory, it suffices to show that $h^1(N_{C/S}(-p_1-\cdots-p_8 - q_1 - \cdots - q_4)) = 0$.  We have
\begin{align*}
h^1(N_{C/S}(-p_1-\cdots-p_8 - q_1 - \cdots - q_4)) &= h^1(\O_C(2H+L)(-p_1-\cdots-p_8 - q_1 - \cdots - q_4)) \\
&=h^1(\O_C(H+L)(p_9- q_1 - \cdots - q_4)) \\
&= h^1(K_C(p_9)(-q_1 - \cdots - q_4)). 
\end{align*}
As $h^1(K_C(p_9)) = 0$, and $q_1, \ldots, q_4$ are general points, the twist $K_C(p_9)(-q_1-\cdots-q_4)$ is still nonspecial as desired.


Finally let us show that there does not exist a curve of degree $9$ and genus $6$ passing through all $14$ points $p_1, \ldots, p_9$ and $q_1, \ldots, q_5$.  The strategy will be similar to the constructive method above, however we will find a smooth del Pezzo containing the first $8$ (general) points $p_1, \ldots, p_8$, and $q_1, \ldots, q_5$; using the fact that any curve of degree $9$ and genus $6$ containing these points must lie on this surface, we show that such a curve cannot pass through $p_9$.

\begin{lem}\label{sm_dP_2}
There is a unique smooth del Pezzo surface of degree $4$ through the points
\[p_1, \ldots, p_8, q_1, \ldots, q_5, \]
that are general subject to the constraint that the $p_i$ lie in a general hyperplane.
\end{lem}
\begin{proof}
First, we show that there is a unique pencil of quadrics passing through $p_1, \ldots, p_8, q_1, \ldots, q_5$.  As there is a $14$-dimensional projective space of quadrics in $\pp^4$, this amounts to asserting that the $13$ points above impose independent conditions on sections of $\O_{\pp^4}(2)$.  Consider the surjective restriction map to $\pp^3$
\[H^0(\O_{\pp^4}(2)) \to H^0(\O_{\pp^3}(2)). \]
As the $8$ points $p_1, \ldots, p_8$ are general in $\pp^3$, they impose independent conditions on sections $H^0(\O_{\pp^3}(2))$ and thus sections of $\O_{\pp^4}(2)$.  The points $q_1, \ldots, q_5$ then impose independent conditions on sections vanishing at $p_1, \ldots, p_8$ as they are general points in $\pp^4$.  

Next, let $S$ be a (smooth) del Pezzo surface and let $p_1, \ldots, p_8$ be a general collection of points on $S \cap H =E$ a general hyperplane section.  Let $q_1, \ldots, q_5$ be $5$ other general points on $S$.  Denote by $\Gamma$ the $0$-dimensional subscheme of the points $p_1 + \cdots + p_8 + q_1 \cdots + q_5$.  We will show that the map from such marked del Pezzo surfaces to $H^8 \times ({\pp^4})^5$ is dominant, by deformation theory, by observing that $N_{S / \pp^4}$ has no higher cohomology and
that the restriction map
\[H^0(N_S) \to H^0(N_S|_\Gamma)\]
is surjective.
Indeed, $N_S \simeq \O_S(2) \oplus \O_S(2)$ is nonspecial as $\O_S(2)$ is ample and so has no higher cohomology by Kodaira vanishing.
To show the surjectivity, note that
the restriction map $H^0(\O_S(2)) \to H^0(\O_E(2))$ is surjective, again for example by Kodaira vanishing, as $E$ is in the class $\O_S(1)$.  Since the points $p_1, \ldots, p_{8}$ are general on $E$, they impose independent conditions on $H^0(\O_E(2))$; therefore they also impose independent conditions on sections of $\O_S(2)$.
The points $q_1, \ldots, q_5$, being general, impose independent conditions on sections of $\O_S(2)$
vanishing at $p_1, \ldots, p_{8}$.
\end{proof}

Now consider any curve $C$ of degree $9$ and genus $6$ in $\pp^4$ passing through $\Gamma$ (by Corollary~\ref{cor_twist} there exist such curves).  By Riemann-Roch, $h^0(\O_C(2)) = 18+1-6 = 13$, and so $C$ lies on at least a pencil of quadrics in $\pp^4$.  As the points of $\Gamma$ lie on a unique, and general, such pencil, $C$ must also lie on a unique such pencil, and so be contained in the base locus $S$, a del Pezzo surface of degree~$4$.  As the Picard group of $S$ is countable, $C$ is in one of countably many possible curve classes on $S$.  Therefore $C \cap H \subset E$ is in one of only countably many linear equivalence classes of degree $9$ divisors on $E$.  This necessarily contains $p_1 + \cdots + p_8$, and so the $9$th point of intersection is in one of only countably many classes on $E$.  In particular, it cannot be a general point on $E$.

\subsection{Degree 10, genus 7}

Let $D = p_1+ \cdots+ p_{9}$ be a collection of points on a general hyperplane section of a curve $C$ of degree $10$ and genus $7$ in $\pp^4$.  We will show that $N_C(-D)$ satisfies interpolation, therefore deducing that $C$ passes through $f(10,7,4) = 14$ points that are general subject to the constraint that $9$ of them lie in a hyperplane.
In this case we have
\[\chi(N_C(-D)) = 17 \equiv -1 \pmod{3}. \]
Therefore it suffices to show that for a general divisor $\Gamma$ of degree $6$, 
\[h^0(N_C(-D-\Gamma)) = 0. \]
As before writing $D = \O_C(1)(-p)$, with $p$ general on $C$, we have that $\O_C(-1)(p-D)$ is general.  Therefore this result again follows from Theorem~\ref{notwist}.

In \cite{quadrics} it was shown that the hyperplane section $H \cap C$ of a curve $C$ of degree $10$ and genus $7$ is a general collection of $10$ points on a quadric in $H \simeq \pp^3$.  We show here that there exists a curve of degree $10$ and genus $7$ through a collection of $f(10,7,4)+1 = 15$ points that are general subject to the condition that $10$ lie on a quadric in $H$. 

We begin in a similar way to our approach for $(9,6)$.  Fix general points $p_1, \ldots, p_{10}$ on a quadric $Q$ in $H \simeq \pp^3$, and $5$ more general points $q_1, \ldots, q_5$ in $\pp^4$.  

\begin{lem}\label{sm_quad}
There exists a unique smooth quadric through the $14$ points 
\[p_1, \ldots, p_9, q_1, \ldots, q_5,\]
that are general subject to the constraint that the $p_i$ lie in a hyperplane.  
(Equivalently through the $15$ points
$p_1, \ldots, p_{10}, q_1, \ldots, q_5$ that are general subject to the 
constraint that the $p_i$ lie in a quadric in a hyperplane.)
\end{lem}
\begin{proof}
As there is a $14$-dimensional projective space of quadrics in $\pp^4$, it suffices to show that these $14$ points impose independent conditions on sections of $H^0(\O_{\pp^4}(2))$.  The points $p_1, \ldots, p_9$ impose independent conditions on $H^0(\O_{\pp^3}(2))$, being general on $\pp^3$.  Furthermore, restriction is a surjection
\[H^0(\O_{\pp^4}(2)) \to H^0(\O_{\pp^3}(2)) \]
so the points $p_1, \ldots, p_9$ impose at least as many conditions on $H^0(\O_{\pp^4}(2))$.  The general points $q_1, \ldots, q_5$ impose independent conditions on sections of $\O_{\pp^4}(2)$ vanishing at $p_1, \ldots, p_9$, and so the whole collection imposes independent conditions on quadrics.

Now we show that this quadric is general.  As above, let $S$ be a general quadric in $\pp^4$ and let $q_1, \ldots, q_5$ be general points on $S$ and $p_1, \ldots, p_9$ be general points on $H \cap S$.  Then it suffices by deformation theory to verify that $N_{S/ \pp^4}$ has no higher cohomology and that
\[H^0(N_S) \to H^0(N_S|_{p_1 \cup p_2 \cup \cdots \cup p_9 \cup q_1 \cup \cdots \cup q_5})\]
is surjective.
To show $N_S \simeq \O_S(2)$ has no higher cohomology,
we use the exact sequence of sheaves
\[0 \to \O_{\pp^4}(n - 2) \to \O_{\pp^4}(n) \to N_S(n) \simeq \O_S(n) \to 0,\]
and the vanishing of the higher cohomology of $\O_{\pp^4}(n)$ and $\O_{\pp^4}(n - 2)$ for $n \geq 1$, which shows that $\O_S(n)$ has no higher cohomology for
$n \geq 1$ (in particular for $n = 2$).

To show that these $14$ points
impose independent conditions on sections of $\O_S(2)$,
we first note that the restriction map $H^0(\O_S(2)) \to H^0(\O_{S \cap H}(2))$
is surjective (since $h^1(\O_S(1)) = 0$ by the above).
Since $p_1, p_2, \ldots, p_9$ are general on $S \cap H$,
and $h^0(\O_{S \cap H}(2)) = h^0(\O_{\pp^1 \times \pp^1}(2, 2)) = 9$,
they impose independent conditions on sections of $\O_{S \cap H}(2)$
and thus on sections of $\O_S(2)$.
Since $q_1, q_2, \ldots, q_5$ are general, they impose independent
conditions on
\[\ker(H^0(\O_S(2)) \to H^0(\O_{S \cap H}(2))) \simeq H^0(\O_S(1)).\]
Combining this, our $14$ points impose independent conditions on sections of $\O_S(2)$, as desired.
\end{proof}

We showed above that there exists a curve $C$ of degree $10$ and genus $7$ through $p_1, \ldots, p_9, q_1, \ldots, q_5$.  As $h^0(\O_C(2)) = 20+1-7 = 14$, the curve $C$ lies on a quadric.  As the $14$ points lie on a unique quadric $S$ (which is smooth if the points are general), $C$ lies on this quadric.

Finally we claim that the rank $2$ bundle $N_{C/S}(-1)$ satisfies interpolation -- this is sufficient to prove the desired result as the points $p_1, \ldots, p_{10}$ are general points in a hyperplane section of this quadric
and $q_1, \ldots, q_5$ are general points on this quadric.  From
\[0 \to N_{C/S}(-1) \to N_C(-1) \to \O_C(1) \to 0, \]
we have $\chi(N_{C/S}(-1)) = 10$.  It suffices to show that after twisting down by $5$ general points $q_1, \ldots, q_5$ on $C$, we have $h^0(N_{C/S}(-1)(-q_1- \cdots-q_5)) = 0$.  By the above exact sequence, it suffices to show that
\begin{equation}\label{10_7_vanishing} h^0(N_{C}(-1)(-q_1-\cdots-q_5)) = 0.\end{equation}

To show this, we will degenerate $C$ to the union of a curve $X$ of degree $6$ and genus $2$ and a degenerate rational normal curve $M \cup_{m_1+m_2+m_3} (L_1 \cup L_2 \cup L_3)$, where $L_i \cap M= m_i$ and $L_i \cap C = z_i + w_i$ is a $2$-secant line as in the setup of Section~\ref{quartic}.  Specialize the points $q_i$ so that $2$ lie on $M$ and one lies on each of the $L_i$.  By \cite[Theorem 1.6]{rbn} this is a BN-curve and as in Lemma~\ref{3_secant} (using Lemma~\ref{3-secant-twist}) we have that on global sections 
\[h^0(N_{C}(-1)(-q_1-\cdots-q_5)) = h^0(N_X(-1)[z_1 \to w_1][w_1 \to z_1][z_2 \to w_2][w_2 \to z_2][z_3 \to w_3][w_3 \to z_3]) .\]
Further degenerate $X$ as in Lemma~\ref{mod_on_secant} to the union of a curve $Y$ of degree $5$ and genus $1$ and a $2$-secant line $L$ meeting $Y$ at points $x$ and $y$.  The points $x_1$ and $x_2$ specialize onto $L$ and the remainder specialize onto $Y$.  The union is a BN-curve and the global sections glue to give that
\begin{align*}
h^0(N_Y(-1)&(y)[x \to y][2y \to x][w_1 \to x][w_2 \to x][z_3 \to w_3][w_3 \to z_3]) \\
& =h^0(N_X(-1)[z_1 \to w_1][w_1 \to z_1][z_2 \to w_2][w_2 \to z_2][z_3 \to w_3][w_3 \to z_3]) \\
& = h^0(N_{C}(-1)(-q_1-\cdots-q_5)).
\end{align*}
The bundle
\[N_Y' \colonequals N_Y(-1)(y)[x \to y][2y \to x][w_1 \to x][w_2 \to x][z_3 \to w_3][w_3 \to z_3]\]
sits in an exact sequence
\[0 \to N_{Y \to x}(-1)(y -x - z_3 - w_3) \to N_Y' \to N_{\pi_x|_Y}(-1)(-y-w_1-w_2)[x \to y][z_3 \to w_3][w_3 \to z_3] \to 0\]
$N_{Y \to x}(-1)(y -x - z_3 - w_3) \simeq \O_Y(x + y - z_3 - w_3)$ is a general line bundle of degree $0$ on an elliptic curve, and hence has no global sections.	
It thus remains to show
\[h^0(N_{\pi_x|_Y}(-1)(-y-w_1-w_2)[x \to y][z_3 \to w_3][w_3 \to z_3]) = 0.\]
Since $\chi(N_{\pi_x|_Y}(-1)(-y-w_1-w_2)[x \to y][z_3 \to w_3][w_3 \to z_3]) = -1$
and $\O_C(1)(y + w_1 + w_2)$ is a general line bundle of degree $2$ on an elliptic curve,
it suffices to show $N_{\pi_x|_Y}[x \to y][z_3 \to w_3][w_3 \to z_3]$ satisfies interpolation.
Limiting $y$ to $w_3$, it suffices to show interpolation for the bundle
\[N_{\pi_x|_Y}[x \to w_3][z_3 \to w_3][w_3 \to z_3]\]
This is verified by publicly available code in \cite[Appendix B]{joint}:

\begin{verbatim}
>>> good(Curve(4,1,3).add(P101).add(P100))
True
\end{verbatim}

\bibliographystyle{plain}
\bibliography{Interpolation}

\end{document}